\documentclass[aop, preprint, 12pt]{imsart}

\RequirePackage[OT1]{fontenc}
\RequirePackage{amsthm,amsmath}
\RequirePackage[numbers]{natbib}
\RequirePackage[colorlinks,citecolor=blue,urlcolor=blue]{hyperref}
\usepackage{amsfonts}
\usepackage{color}
\usepackage{pifont}
\usepackage{amssymb}
\usepackage[english]{babel}
\usepackage[T1]{fontenc}
\usepackage{graphicx}
\usepackage{subfigure}
\usepackage{latexsym}
\usepackage{amstext}
\usepackage{mathrsfs}
\usepackage{hyperref}
\usepackage{dsfont}
\usepackage{graphics}
\usepackage{enumerate}
\usepackage{paralist}
\usepackage{color}
\usepackage{fullpage}

\numberwithin{equation}{section}

\newcommand{\beq}{\begin{eqnarray}}
\newcommand{\beqq}{\begin{eqnarray*}}
\newcommand{\eeq}{\end{eqnarray}}
\newcommand{\eeqq}{\end{eqnarray*}}
\newcommand{\eps}{\varepsilon}
\usepackage{mathtools}
\usepackage[english]{babel}
\usepackage[utf8]{inputenc}
\usepackage{amsmath}
\usepackage{graphicx}
\usepackage[colorinlistoftodos]{todonotes}
\usepackage{polynom}
\usepackage{amsfonts}
\usepackage{dsfont}
\usepackage{amsthm}
\usepackage{hyperref}
\usepackage{graphicx}
\newtheorem{theorem}{Theorem}[section]

\newtheorem{lemma}{Lemma}[section]
\newtheorem{proposition}[theorem]{Proposition}
\newtheorem{corollary}[theorem]{Corollary}

\usepackage{fullpage}
\usepackage[T1]{fontenc}
\usepackage{hyperref}
\usepackage{xcolor}
\definecolor{link-color}{rgb}{0.15,0.4,0.15}
\hypersetup{
  colorlinks,
  linkcolor = {link-color}, citecolor = {link-color},
}

\usepackage{graphicx}
\usepackage{hyperref}
\usepackage{amsmath,amsthm,amssymb}
\usepackage{bbm}
\usepackage{tabularx}

\newcommand{\R}{\mathbb{R}}

    \def\d{{\textnormal d}}

\newenvironment{eqnarr}{\begin{IEEEeqnarray}{rCl}}{\end{IEEEeqnarray}\ignorespacesafterend}

\newcommand{\cenX}{\stackrel{_\leadsto}{X}}
\newcommand{\cenpi}{\stackrel{_\leadsto}{\pi}}
\renewcommand{\eqref}[1]{\hyperref[#1]{(\ref*{#1})}}

    \def\beq{\begin{eqnarr}}
    \def\eeq{\end{eqnarr}}
    \def\beqq{\begin{eqnarray*}} 
    \def\eeqq{\end{eqnarray*}} 

        \def\d{{\rm d}}

    \def\d{{\textnormal d}}
\newtheorem{remark}{Remark}[section]

\newcommand*{\pref}[1]{\hyperref[#1]{(\ref*{#1})}}
\newcommand*{\refpref}[2]{\hyperref[#2]{\ref*{#1}(\ref*{#2})}}

\startlocaldefs
\numberwithin{equation}{section}
\theoremstyle{plain}

\endlocaldefs

\begin{document}

\begin{frontmatter}
\title{Stable L\'evy processes in a cone}

\runtitle{Recurrent extension for ssMp in a Wedge}

\begin{aug}

\author{\fnms{Andreas E. Kyprianou}\thanksref{t3}\ead[label=e1]{a.kyprianou@bath.ac.uk}},
\author{\fnms{Victor Rivero}\thanksref{t1}\ead[label=e2]{rivero@cimat.mx}}
\and
\author{\fnms Weerapat Satitkanitkul\thanksref{t2}
\ead[label=e3]{weerapat.satit@gmail.com}}

\thankstext{t1}{Supported by EPSRC grant EP/M001784/1}

\thankstext{t3}{Supported by EPSRC grants EP/L002442/1 and EP/M001784/1}

\thankstext{t2}{Supported by a Royal Thai PhD scholarship}

\affiliation{University of Bath, CIMAT, Universit\'e d'Angers}

\address{
A.E. Kyprianou and W. Satitkanitkul\\
University of Bath\\
Department of Mathematical Sciences \\
Bath, BA2 7AY\\
 UK.\\
\printead{e1}
}

\address{V. Rivero\\
CIMAT A. C.,\\
Calle Jalisco s/n,\\
Col. Valenciana,\\
A. P. 402, C.P. 36000,\\
Guanajuato, Gto.,\\
Mexico.\\
\printead{e2}
}

\address{W. Satitkanitkul\\
 	LAREMA,\\
		Universit\'e d'Angers,\\
		2, Boulevard Lavoisier, \\
		49045 Angers cedex 01,\\
		France.\\
		 \printead{e3}
}

\end{aug}

\begin{abstract}\hspace{0.1cm}
Ba\~nuelos and Bogdan \cite{BB} and  Bogdan et al. \cite{BPW} analyse the asymptotic tail distribution of the first time a stable (L\'evy) process in dimension $d\geq 2$ exits a cone. 
We use these results to develop the notion of a stable process conditioned to remain in a cone as well as the the notion of a stable process conditioned to absorb continuously at the apex of a cone (without leaving the cone). As self-similar Markov processes we examine some of their fundamental properties through the lens of its Lamperti--Kiu decomposition. 
In particular we are interested to understand the underlying structure of the Markov additive process that drives such processes. As a consequence of our interrogation of the underlying MAP, we are able to provide an answer by example to the open question: If the modulator of a MAP has a stationary distribution, under what conditions does its ascending ladder MAP have a stationary distribution?

 With the  help of an analogue of the Riesz--Bogdan--\.Zak transform (cf. Bogdan and \.Zak \cite{BZ}, Kyprianou \cite{Deep1}, Alili et al. \cite{ACGZ}) as well as Hunt-Nagasawa duality theory, we show how the two forms of conditioning are dual to one another. Moreover, in the sense of Rivero \cite{R05, R07} and Fitzsimmons \cite{Fitz06}, we construct the null-recurrent extension of the stable process killed on exiting a cone, showing that it again remains in the class of self-similar Markov processes. Aside from the Riesz--Bogdan--\.Zak transform and Hunt-Nagasawa duality,  an unusual combination of the Markov additive renewal theory of e.g. Alsmeyer \cite{Alsmeyer1994} as well as the boundary Harnack principle (see e.g. \cite{BPW}) play a central role to the analysis.

In the spirit of several very recent works (see \cite{KRS, Deep1, Deep2, Deep3, ALEAKyp, DK}), the results presented here show that many previously unknown results of stable processes, which have long since been understood for Brownian motion, or are easily proved for Brownian motion, become accessible by appealing to the notion of the stable process as a self-similar Markov process, in addition to its special status as a  L\'evy processes with a semi-tractable potential analysis.
\end{abstract}



\end{frontmatter}

\section{Introduction}
For $d\geq 2$, let $X:= (X_t:t\geq 0)$, with probabilities $\mathbb{P}=(\mathbb{P}_x, x\in\mathbb{R}^d)$,  be a $d$-dimensional isotropic stable process  of index $\alpha\in(0,2)$. That is to say, $(X, \mathbb{P})$ is a $\mathbb{R}^d$-valued L\'evy process having  characteristic triplet $(0,0,\Pi)$, where
\begin{equation}\label{jumpmeasure}
\Pi(B)=\frac{2^{\alpha}\Gamma(({d+\alpha})/{2})}{\pi^{d/2}|\Gamma(-{\alpha}/{2})|}\int_{B}\frac{1}{|y|^{\alpha+d}}{\rm d}y , \qquad B\in\mathcal{B}(\mathbb{R}).
\end{equation}
Equivalently, this means $(X, \mathbb{P})$ is a $d$-dimensional  L\'evy process with characteristic exponent $\Psi(\theta) = -\log\mathbb{E}_0({\rm e}^{{\rm i}<\theta, X_1>})$ which satisfies
\[
\Psi(\theta) = |\theta|^\alpha, \qquad \theta\in\mathbb{R}^d.
\]

Stable processes are also self-similar in the sense that they satisfy a scaling property. More precisely, for $c>0$ and $x\in\mathbb{R}^d\setminus\{0\}$,
\begin{equation}
\text{under }
 \mathbb{P}_x, \text{ the law of }(cX_{c^{-\alpha}t}, t\geq 0) \text{ is equal to } \mathbb{P}_{cx}.  
 \label{1/a}
 \end{equation}
As such, stable processes are useful prototypes for the study of  the class of L\'evy processes and, more recently, for the study of the class of self-similar Markov processes. The latter class of processes are regular strong Markov processes which respect the scaling relation \eqref{1/a}, and accordingly are identified as having self-similarity (Hurst) index $1/\alpha$.

\medskip

In this article, we are interested in understanding the notion of conditioning such stable processes to remain in a Lipschitz cone,
\begin{equation}
\Gamma =\{x\in \mathbb{R}^d :x\neq 0 , \arg(x)\in\Omega\},
\label{GammaOmega}
\end{equation}
where  $\Omega$ is open on $\mathbb{S}^{d-1}: = \{x\in\mathbb{R}^d: |x| =1\}$. Note that $\Gamma$ is an open set which does not include its apex $\{0\}$, moreover, $\Omega$ need not be a connected domain. See \cite{BPW} for the notion of Lipschitz cone and some related facts. 
\smallskip

Our motivation comes principally from the desire to show how the rapidly evolving theory of self-similar Markov processes presents a number of new opportunities to contextualise existing theory and methodology in a completely new way to attack problems, which may have otherwise been seen as beyond reach. We note in this respect that the key tool, the Lamperti--Kiu transform for self-similar Markov processes, was formalised only recently in \cite{CPR, KKPW, ACGZ}. It   identifies self-similar Markov processes as in one-to-one correspondence with Markov additive processes through a generalised polar decomposition with additional time change, and is the principal tool which,  in the last five years or so, has unlocked a number of ways forward for classical problems such as the one we consider here; see \cite{CPR, KRS, KPW, KV, Deep1, Deep2, Deep3, DDK, DK, DKW}. 
Moreover, this  new perspective opens up an entire new set of challenges both in the setting of the underlying class of Markov additive processes (which have seldom received attention in the general setting since the foundational work of e.g. \c{C}inlar \cite{Cinlar, Cinlar2, Cinlar1} and Kaspi \cite{Kaspi}) as well as the general class of self-similar Markov processes. Many of these challenges also emerge naturally in the setting of other stochastic processes and random structures where self-similarity plays an inherently fundamental role; see for example \cite{StR} in the setting of multi-type fragmentation processes, \cite{BWat} in the setting of growth fragmentation processes and \cite{BBCK, BCK} in the setting of planar maps. In this respect interrogating fundamental questions in the stable setting lays the foundations to springboard to problems of significantly greater generality. We mention in this respect, an outstanding problem in the setting of stable L\'evy processes, which relates to the extremely deep work of e.g. \cite{Weyl, Weyl2, Weyl4} which showed how to condition a Brownian motion to stay in a Weyl Chamber and the important relationship this has to random matrix theory. This also inspired similar conditionings of other processes, such as those that appear in queueing theory; see \cite{Weyl3}.

\smallskip

Our journey towards conditioning stable processes to remain in the cone $\Gamma$ will take us through a number of striking relations between stable processes killed on exiting  $\Gamma$ and stable processes conditioned either enter or to to absorb continuously at the apex of $\Gamma$, which are captured by space-time transformations. Our analysis will necessitate examining new families of Markov additive processes  that underly the conditioned stable processes through the Lamperti--Kiu transform. As an example we will exhibit a rare and difficult  result  which identifies semi-explicitly the existence of a stationary distribution for the at radially extreme points of the conditioned process (showing how the harmonic functions that drive our conditionings influence the strong mixing of the angular process). 
Moreover, our work will complement a number of other  works which have examined the notion entrance laws of self-similar Markov processes, as well as the subsequent notion of recurrent extension. Specifically, in the setting of Brownian motion and stable processes in a Lipschitz cone, this pertains to \cite{BDS, BPW}, and more generally for self-similar Markov processes,  see   \cite{BY02, CC06b, R05, R07, Fitz06, DDK}. Our approach will consist of an unusual mixture of techniques, coming from potential analysis and Harnack inequalities, Markov additive renewal theory,  last exit decompositions {\it \`a la} Maisonneuve and It\^o synthesis.

\section{Harmonic functions in a cone} For simplicity, we denote $\kappa_\Gamma$ be the exit time from the cone i.e. 
$$\kappa_\Gamma =\inf\{s>0: X_s \notin \Gamma\}.$$
Ba\~nuelos and Bogdan \cite{BB} and  Bogdan et al. \cite{BPW} analyse the tail behaviour of the stopping time $\kappa_\Gamma$. Let us spend a moment reviewing their findings. 

\medskip

Suppose that we write 
\[
U_\Gamma(x, \d y) = \int_0^\infty  \mathbb{P}_x(X_t \in \d y, \, t<\kappa_\Gamma)\d t, \qquad x, y\in\Gamma
\]
for the potential of the stable process killed on exiting $\Gamma$.

\medskip

Then it is known e.g. from Section 2 of \cite{BPW} that $U_\Gamma(x,\d y)$ has a density, denoted by $u_\Gamma(x,y)$ and that 
\[
M(y): = \lim_{|x|\to\infty}\frac{u_\Gamma(x, y)}{ u_\Gamma(x, y_0)}, \qquad y\in\Gamma,
\]
exists and depends on  $y_0\in\Gamma$ only through a normalising constant. Note that it is a consequence of this definition that $M(x) =0$ for all $x\not\in\Gamma$.
Moreover, 
 $M$ is locally bounded on $\mathbb{R}^d$ and homogeneous of degree $
 \beta = \beta(\Gamma, \alpha)\in(0,\alpha)$, meaning, 
\begin{equation}
M(x) = |x|^\beta M(x/|x|) = |x|^\beta M(\arg(x)) , \qquad x\neq 0.
\label{scaling}
\end{equation}
It is also known that, up to a multiplicative constant,  $M$ is the unique function which is harmonic in the sense that 
\begin{equation}
M(x) = 
\mathbb{E}_x[M(X_{\tau_B})\mathbf{1}_{( \tau_B<\kappa_\Gamma)}], \qquad x\in \mathbb{R}^d,
\label{martingale}
\end{equation}
where $B$ is any open bounded domain and $\tau_B = \inf\{t>0: X_t \not\in B\}$.

\medskip

The function $M$ plays a prominent role in the following asymptotic result in Corollary 3.2 of  Bogdan et al. \cite{BPW}, which strengthens e.g.  Lemma 4.2 of Ba\~nuelos and Bogdan \cite{BB}.

\begin{proposition}[Bogdan et al. \cite{BPW}]\label{BPWprop}
We have 
\[
\lim_{a\to0}\sup_{x\in \Gamma , \, |t^{-1/\alpha}x|\leq a}\frac{\mathbb{P}_x(\kappa_\Gamma > t)}{M(x)t^{-\beta/\alpha}} =C ,
\]
where $C>0$ is a constant.
\end{proposition}

Before moving to the next section we also mention the earlier works of DeBlaissie  \cite{DeB} and  M\'endez-Hern\'andez \cite{MH}, who considered moment properties of the exit time from the cone.

\section{Results for stable processes conditioned to stay in a cone}
The above summary of the results in \cite{BB, BPW}  will allow us to introduce the notion of stable process conditioned to stay in $\Gamma$. Before doing that we make a slight digression to introduce some notation. 

Let $\mathbb{D}$ be the space of c\`adl\`ag paths defined on
$[0,\infty)$, with values in $\R^{d}\cup\Delta$, where $\Delta$ is a
cemetery state. Each path $\omega\in\mathbb{D}$ is such that
$\omega_t=\Delta$, for any $t\ge\inf\{s\geq
0:\omega_s=\Delta\}=:\zeta(\omega)$. As usual we extend any function
$f:\mathbb{R}^d\to\mathbb{R}$ to $\mathbb{R}^d\cup\Delta$ by taking $f(\Delta)=0.$  The space
$\mathbb{D}$ is endowed with the Skohorod topology and its Borel
$\sigma$-field. Here after we will denote by $X$ the canonical process of the coordinates, and by $\left(\mathcal{F}_t,\, t\geq 0 \right)$ the right-continuous filtration generated by $X$. We will also denote by $\mathbb{P}^{\Gamma}$ the law of the stable process $(X,\mathbb{P})$ killed when it leaves the cone $\Gamma.$ Note, in particular, that $(X, \mathbb{P}^{\Gamma})$ is also a self-similar Markov process. 

\begin{theorem}\label{conditioned}$ \mbox{ }$
\begin{itemize}
\item[(i)]For any $t>0,$ and $x\in \Gamma$,
\[
\mathbb{P}^{\triangleleft}_x(A) :=\lim_{s\to \infty}\mathbb{P}_x\left( A\, | \kappa_\Gamma >t+s \right), \qquad A\in\mathcal{F}_t,
\]
defines a family of conservative probabilities on the space of c\`adl\`ag paths
which respect the Doob $h$-transform
\begin{equation}
\left.\frac{\d\mathbb{P}^\triangleleft_x}{\d\mathbb{P}_x}\right|_{\mathcal{F}_t
}: =  \mathbf{1}_{(t<\kappa_\Gamma
)}\frac{M(X_t)}{M(x)}, \qquad t\geq 0,\ \text{and}\ x\in \Gamma.
\label{MCOM}
\end{equation}
In particular, the right-hand side of \eqref{MCOM} is a martingale.
\item[(ii)] Let $\mathbb{P}^\triangleleft:=\left(\mathbb{P}^\triangleleft_x, x\in \Gamma\right).$ The process $(X, \mathbb{P}^\triangleleft)$, is a self-similar Markov process.
\end{itemize}
\end{theorem}


 Next, we want to extend the definition of the process $(X, \mathbb{P}^{\triangleleft})$, to include the apex of the cone $\Gamma$ as a point of issue in a similar spirit to the inclusion of the origin as a point of issue for positive and real-valued self-similar Markov processes (cf. Bertoin and Yor \cite{BY02}, Bertoin and Caballero \cite{BC02},   Caballero and Chaumont  \cite{CC06a}, Bertoin and Savov \cite{BS}, Chaumont et al. \cite{CKPR12}, Dereich et al. \cite{DDK}). Said another way, we want to show the consistent inclusion of the state $0$ to the state space $\Gamma$ in the definition of $(X, \mathbb{P}^\triangleleft)$ as both a self-similar and a Feller process.
 
 \smallskip
 
Before stating our theorem in this respect, we must first provide a candidate law for $\mathbb{P}^\triangleleft_0$, which is consistent with the family $\mathbb{P}^\triangleleft$. To this end, we need to recall the following theorem which is a copy of Theorem 3.3 in  Bogdan et al. \cite{BPW} (see also \cite{BDS, DeB} for earlier work pertaining to distributional identities concerning the entrance law from the apex of Brownian motion in a cone). In order to state it, we need to introduce the notation $p^\Gamma_t(x, y)$, $x,y\in\Gamma$, $t\geq 0$, which will denote the transition density of $(X, \mathbb{P}^{\Gamma})$. Note the reason why this density exists is because of the existence of a transition density for $X$, say $p_t(x,y)$, $x,y\in\mathbb{R}^d$, $t\geq 0$ and the relation
\[
p^\Gamma_t(x,y):=p_t(x,y)-\mathbb{E}_x[\kappa_\Gamma<t ; p_{t-\kappa_\Gamma}(X_{\kappa_\Gamma},y)],\qquad  x, y\in\Gamma, t>0.
\]

\begin{theorem}[Bogdan et al. \cite{BPW}]\label{entrance}
The following limit exits,
\begin{equation}
n_t(y):= \lim_{\Gamma\ni x\rightarrow0}\frac{p^\Gamma_t(x,y)}{\mathbb{P}_x(\kappa_\Gamma>t)t^{\beta/\alpha}}, \qquad x, y\in\Gamma, t>0,
\label{P00}
\end{equation}
and  $(n_t(y)\d y, t>0)$, serves as an entrance law to $(X,\mathbb{P}^{\Gamma})$, in the sense that 
\[
n_{t+s}(y) = \int_\Gamma n_t(x)p^\Gamma_s(x,y)\d x, \qquad y\in \Gamma, s,t\geq 0.
\]
The latter is a finite strictly positive jointly continuous function  with the properties
\begin{equation}
n_t(y) = t^{-(d+\beta)/\alpha}n_1(t^{-1/\alpha}y) \,\,\text{ and }\,\,
 n_1(y) \approx 
 \frac{\mathbb{P}_y(\kappa_\Gamma >1)}{(1 + |y|)^{d+\alpha} }, \qquad y\in \Gamma, t>0,
 \label{scalen}
\end{equation}
and 
\begin{equation}
\int_\Gamma n_t(y)\d y = t^{-\beta/\alpha}, \qquad t>0,
 \label{scalen2}
\end{equation}
where $ f\approx g$ means the ratio of the functions $f$ and $g$ are bounded from above and below by two positive constants, uniformly in their domains. 
\end{theorem}

The existence of the entrance law $(n_t(y)\d y, t>0),$ is sufficient to build  a candidate for a probability measure, say $\mathbb{P}^\triangleleft_0,$ on $\mathbb{D}$ carried by the paths of $X$ that start continuously from $0$ and remain in $\Gamma$ forever. 
 To that end, note  from \eqref{P00} and Proposition \ref{BPWprop}, that for any $t>0$ we have the following weak convergence,
\begin{align}
\lim_{\Gamma\ni x\to0}\frac{M(y)}{M(x)}\mathbb{P}_x(X_t \in \d y , \, t<\kappa_\Gamma) 
&= \lim_{\Gamma\ni x\to0}C\frac{M(y)}{\mathbb{P}_x(t<\kappa_\Gamma)t^{\beta/\alpha} }p^\Gamma_t(x,y)\d y\notag\\
&=CM(y)n_t(y)\d y,
\qquad y\in \Gamma,
\label{P0}
\end{align}
where $C$ is the constant in Proposition \ref{BPWprop}, and it is independent of $t$.   With a pre-emptive choice of notation, let us denote by $\mathbb{P}^\triangleleft_0(X_t \in \d y)$ the measure that is obtained as a limit in the above relation, that is $$\mathbb{P}^\triangleleft_0(X_t \in \d y):=CM(y)n_t(y)\d y,\qquad y\in \Gamma, t>0.$$ Recalling from \eqref{MCOM} that  $M$ forms a  martingale for the process killed at its first exit from $\Gamma,$ we have that necessarily $\int_{\Gamma}\mathbb{P}^\triangleleft_0(X_t \in \d y)=1.$ Furthermore, denote by $\mathbb{P}^\triangleleft_0$ the probability measure on $\mathbb{D}$ whose finite dimensional distributions are given by, 
\[
\mathbb{P}_0^\triangleleft(X_{t_i}\in A_i, \, i = 1,\cdots, n) := 
C\int_{A_1}M(y)n_{t_1}(y) \mathbb{P}_y^\triangleleft(X_{t_i-t_1}\in A_i, \, i = 2,\cdots, n)\,\d y,
\] 
for $n\in\mathbb{N}$, $0<t_1\leq \cdots t_n<\infty$ and Borel subsets of $\Gamma,$ $A_1, \cdots, A_n$.
The weak convergence in (\ref{P0}) extends in a straightforward way to the finite dimensional convergence $\mathbb{P}^{\triangleleft}_{x}\xrightarrow[\Gamma\ni x\to0]{\text{f.d.}}\mathbb{P}_0^\triangleleft.$ Our main result in this respect establishes that the convergence holds in the stronger sense of Skorohod's topology. 




\begin{theorem}\label{0}
The limit $\mathbb{P}_0^\triangleleft : =\lim_{\Gamma\ni x\to0}\mathbb{P}^\triangleleft_x$ is well defined on the Skorokhod space, so that, $(X, (\mathbb{P}^\triangleleft_x, x\in\Gamma\cup\{0\}))$ is both Feller and self-similar which enters continuously at the origin, after which it never returns and subsequently remains in $\Gamma$.
\end{theorem}

The proof of Theorem  \ref{0} leads us to a stronger understanding how stable processes enter into the wedge $\Gamma$ from its apex, specifically showing that there is functional continuity on the Skorokhod space, complementing the existing works of \cite{BPW} in the stable setting and \cite{BDS, DeB} in the Brownian setting, which deal with specific functionals of the respective processes killed on exiting a cone.  
 \smallskip

Recalling Williams' classical decomposition of Brownian excursions from the origin, which shows that their entrance law can be constructed from the Brownian motion conditioned to stay positive (also identifiable as a Bessel-3 process), one may wonder, in light of Theorem \ref{0}, if it there is a self-similar process that behaves like $(X, \mathbb{P}^{\Gamma})$, but once it exits $\Gamma$ it is not absorbed but returns $0$ in a sensible way. In the specialised literature, a process bearing those characteristics would be called a {\it self-similar recurrent extension} of the stable process killed on exiting the cone $\Gamma,$ $(X, \mathbb{P}^{\Gamma}),$ i.e. a $\Gamma\cup\{0\}$--valued process that behaves like $(X,\mathbb{P}^\Gamma)$ up to the first hitting time of $\partial \Gamma$, for which $0$ is a recurrent and regular state, and that has the scaling property (\ref{1/a}). If such a process exists, say $({\cenX}_t,t\geq 0),$ the fact that $0$ is regular for it, implies that there exists a local time at $0$, say $L$, and an excursion measure from $0,$ say $\mathbf{N}^{\Gamma}.$ We will see in Section \ref{proofrecex} that either $\mathbf{N}^{\Gamma}(X_{0+}\neq 0)=0$ or $\mathbf{N}^{\Gamma}(X_{0+}=0)=0$. In the former case, we say that the recurrent extension leaves $0$ continuously, and in the latter that it leaves $0$ by a jump. General results from excursion theory (see e.g. \cite{fitzsimmons-getoor, Blum1983}),  ensure that both objects together $(L, \mathbf{N}^{\Gamma})$ characterize ${\cenX}.$ Furthermore, the measure $\mathbf{N}^{\Gamma}$ is a
 {\it self-similar excursion measure compatible} with the transition semi-group of $(X, \mathbb{P}^{\Gamma}),$ that is, it is a measure on $(\mathbb{D}, \mathcal{F})$ such that 
\begin{description}
\item[(i)] it is carried by the set of $\Gamma$-valued paths that are sent to the cemetery state $0$ at their death with lifetime $\zeta,$ 
\[ 
\{\chi\in \mathbb{D}\, |\,  \zeta>0, \chi_s\in \Gamma, s<\zeta, \text{ and } \chi_t=0, \text{ for all } t\geq \zeta\},
\]
i.e. where, tautologically, $\zeta = \inf\{t>0: \chi_t=0\}$;
\item[(ii)] the Markov property under $\mathbf{N}^{\Gamma},$ is satisfied, that is, for every bounded $\mathcal{F}$-measurable variable $Y$ and $A\in \mathcal{F}_t$, $t>0$,
$$ \mathbf{N}^{\Gamma}(Y\circ \theta_t, A\cap \{t<\zeta\})=\mathbf{N}^{\Gamma}\left(\mathbb{E}^{\Gamma}_{X_t}[Y], A \cap \{t<\zeta\}\right);$$
\item[(iii)] the quantity $\mathbf{N}^{\Gamma}(1-{\rm e}^{-\zeta})$ is finite;
\item[(iv)] there exists a $\gamma\in (0,1)$ such that for any $q,c>0$,
\begin{equation}
 \mathbf{N}^{\Gamma}\left(\int_0^{\zeta} {\rm e}^{-qs} f(X_s) {\rm d}s\right)= c^{(1-\gamma)\alpha}\mathbf{N}^{\Gamma}\left(\int_0^{\zeta} {\rm e}^{-qc^{\alpha}s} f(cX_s){\rm d} s\right).
 \label{potscale}
\end{equation}
\end{description}
The condition (iv) above is equivalent to require that
\begin{description}
\item[(iv')]  
 there exists $\gamma \in(0,1)$ such that, for any $c>0$ and $f:\Gamma\to \mathbb{R}^+$ measurable, 
$$  \mathbf{N}^{\Gamma}(f(X_s),s<\zeta) =c^{\alpha \gamma}\mathbf{N}^{\Gamma}(f(c^{-1} X_{c^{\alpha }s}), c^{\alpha} s<\zeta),\, \text{ for }s>0$$
\end{description}
The fact that  $\mathbf{N}^{\Gamma}$ necessarily satisfies the above conditions is a consequence of  a straightforward extension of the arguments in Section 2.2 in \cite{R05}. 

\smallskip

Conversely, It\^o's synthesis theorem ensures that, given an excursion measure satisfying the conditions (i)-(iv) above, and a local time at zero, there is a self-similar recurrent extension of $(X, \mathbb{P}^{\Gamma}).$ Using this fact, and that the entrance law $(n_t(y)\d y, t>0)$ is intimately related to an excursion measure,  we establish in the next result the existence of  unique self-similar recurrent extension of $(X, \mathbb{P}^{\Gamma})$ that leaves $0$ (the apex of the cone) continuously.  Furthermore, we will give a complete description  of recurrent extensions that leave zero by a jump.   


\begin{theorem}\label{recex}
Let $\mathbf{N}^{\Gamma}$ be a self-similar excursion measure compatible with $(X,\mathbb{P}^{\Gamma}).$ We have that there exists a $\gamma\in (0, \beta/\alpha)$, a constant $a\geq 0$, and a measure $\pi^{\Gamma}$ on $\Omega$ such that $a\pi^{\Gamma}\equiv 0,$ $\int_{\Omega}\pi^{\Gamma}(\d\theta)M(\theta)<\infty,$ and  $\mathbf{N}^{\Gamma}$ can be represented by, for any $t>0,$ and any $A\in\mathcal{F}_t$ 
\begin{equation}\label{EL-decomp3}
\begin{split}
\mathbf{N}^{\Gamma}(A \,,\,  t<\zeta) =a\mathbb{E}^{\triangleleft}_0\left[\frac{1}{M(X_t)}\mathbf{1}_A\right] +\int_0^\infty\frac{\d r}{r^{1+\alpha\gamma}}\int_{\Omega}\pi^{\Gamma}(\d\theta)\mathbb{E}_{r\theta}[A \, ,\, t<\kappa_\Gamma].
\end{split}
\end{equation}  
If $a>0,$ the process $(X, \mathbb{P}^{\Gamma})$ has unique recurrent extension that leaves $0$ continuously, and $\gamma=\beta/\alpha$. 
\smallskip

Conversely, for each $\gamma\in(0,\beta/\alpha)$, and $\pi^{\Gamma}$ a non-trivial measure satisfying the above conditions,  there is a unique recurrent extension that leaves zero by a jump and such that 
$$\mathbf{N}^{\Gamma}\left(|X_{0+}|\in \d r, \arg(X_{0+})\in \d \theta\right)=\frac{\d r}{r^{1+\alpha\gamma}}\pi^{\Gamma}(\d \theta), \qquad r>0, \theta\in \Omega.$$

\smallskip

Finally, any self-similar recurrent extension with excursion measure $\mathbf{N}^{\Gamma}$ has an invariant measure 
\begin{equation*}
\begin{split}
\cenpi\!{}^\Gamma (\d x) &:=\mathbf{N}^{\Gamma}\left(\int^{\zeta}_{0}\mathbf{1}_{(X_t\in \d x)}\d t\right)
\\
&= a|x|^{\alpha -d-\beta}M(\arg(x))\d x+\int_0^\infty\frac{\d r}{r^{1+\alpha\gamma}}\int_{\Omega}\pi^{\Gamma}(\d\theta)\mathbb{E}_{r\theta}\left[\int^{\kappa_\Gamma}_{0}\mathbf{1}_{(X_t\in \d x)}\d t\right],
\end{split}
\end{equation*}
which is unique up to a multiplicative constant, and this measure is sigma-finite but not finite. 
\end{theorem}

As a final remark, we note that, whilst we have provided a recurrent extension from the apex of the cone, one might also consider the possibility of a recurrent extension from the entire boundary of the cone. We know of no specific work in which there is a recurrent extension from a set rather than a point. That said, one may consider the work on censored stable processes as meeting this notion in some sense; see e.g. \cite{BBC, KPW}

\section{Auxiliary results for associated MAPs}\label{radial}

 \smallskip
 
By an $\mathbb{R}\times\mathbb{S}^{d-1}$-valued Markov additive process (MAP), we mean here that  $(\xi,\Theta) =( (\xi_t, \Theta_t), t\geq 0)$  is a regular Strong Markov Process on $\mathbb{R}\times\mathbb{S}^{d-1}$ (possibly with a cemetery state),  with probabilities $\mathbf{P}:=(\mathbf{P}_{x,\theta}, x\in\mathbb{R}, \theta\in\mathbb{S}^{d-1})$,  such that, for any $t\geq 0$, the conditional law of the process $((\xi_{s+t}-\xi_t,\Theta_{s+t}):s\geq 0)$, given  $\{(\xi_u,\Theta_u), u\leq t\},$ is that of $(\xi,\Theta)$ under $\mathbf{P}_{0,\theta}$, with $\theta=\Theta_t$. For a MAP pair $(\xi, \Theta)$, we call $\xi$ the {\it ordinate} and $\Theta$ the {\it modulator}.
 \smallskip

A very useful fact in the theory of self-similar Markov process is the so called Lamperti-Kiu transform, which is one of the main results in \cite{ACGZ}, extending the seminal work of Lamperti in \cite{L72}, and establishes that there is a bijection between self-similar Markov processes (ssMp) in $\mathbb{R}^d,$  and $\mathbb{R}\times\mathbb{S}^{d-1}$-valued MAPs. Indeed, for any ssMp in $\mathbb{R}^d$, say $Z,$ there exists a unique $\mathbb{R}\times\mathbb{S}^{d-1}$-valued MAP $(\xi,\Theta)$ such that $Z$ can be represented as 
   \begin{equation}\label{eq:lamperti_kiu}
    Z_t  = \begin{cases}\exp\{\xi_{\varphi(t)}\}\Theta_{\varphi(t)}, & t < I_\infty,\\  
    \Delta, & t\geq I_{\infty}, \end{cases}
   \end{equation}
  where $I_t := \int_0^t {\rm e}^{\alpha\xi_t}\d t
$, $t\geq 0$ (so that $I_\infty$ is the almost sure monotone limit $  =\lim_{t\to\infty}I_t$) and
   \begin{equation}
 \label{varphi}
 \varphi(t) = \inf\{s>0: \int_0^s{\rm e}^{\alpha\xi_u}\d u>t\}, \qquad t\geq 0.
 \end{equation}
 Reciprocally, given a $\mathbb{R}\times\mathbb{S}^{d-1}$-valued MAP $(\xi,\Theta)$ the process defined in (\ref{eq:lamperti_kiu}) is an $\mathbb{R}^d$-valued ssMp. This is known as the Lamperti-Kiu representation of the ssMp $Z$. When the lifetime of $Z$ is infinite a.s. we say that law is conservative, which is equivalent to require that $I_\infty = \infty$ almost surely.

\smallskip

In Theorem 3.13 of \cite{ALEAKyp}, it was shown that the MAP underlying the stable process, for whom we will henceforth reserve the notation $\mathbf{P}=(\mathbf{P}_{x,\theta}, x\in\mathbb{R}, \theta\in\mathbb{S}^{d-1})$ for its probabilities,  is a pure jump process, such that $\xi$ and $\Theta$ jump simultaneously. Moreover, the instantaneous jump rate with respect to Lebesgue time $\d t$ when $(\xi_t, \Theta_t) = (x,\vartheta)$ is given by 
\begin{equation}
c(\alpha)\frac{{\rm e}^{yd}}{|{\rm e}^{y}\phi-   \vartheta|^{\alpha+d}}\d y\sigma_1(\d\phi), \qquad t\geq 0,
\end{equation}
where $\sigma_1(\phi)$ is the surface measure on $\mathbb{S}^{d-1}$ normalised to have unit mass and 
\[
c(\alpha) =2^{\alpha-1}\pi^{-d}\frac{\Gamma((d+\alpha)/2)\Gamma(d/2)}{\big|\Gamma(-\alpha/2)\big|}.
\]
More precisely, suppose that  $f$ is a bounded measurable function on $(0,\infty)\times\mathbb{R}^2\times\mathbb{S}^{d-1}\times\mathbb{S}^{d-1}$ such that $f(\cdot, \cdot, 0,\cdot, \cdot) = 0$, then, for all $\theta\in\mathbb{S}^{d-1}$,
\begin{equation}
\begin{split}
&\mathbf{E}_{0,\theta}\left(\sum_{s>0} f(s,\xi_{s-},\Delta\xi_{s},\Theta_{s-}, \Theta_s)\right)\\
&=\int_{(0,\infty)\times\mathbb{R}\times\mathbb{S}^{d-1}}V_\theta(\d s, \d z, \d \vartheta)\int_{\mathbb{S}^{d-1}}\int_{\mathbb{R}}\sigma_1(\d\phi)\d y\frac{c(\alpha){\rm e}^{yd}}{|{\rm e}^{y}\phi-   \vartheta|^{\alpha+d}}f(s, z, y,\vartheta, \phi),
\end{split}
\end{equation}
where $\Delta\xi_s = \xi_s - \xi_{s-}$,
\[
V_\theta(\d t, \d z, \d \vartheta) = \mathbf{P}_{0,\theta}(\xi_t\in \d z, \Theta_t\in \d\vartheta)\d t , \qquad  z\in \mathbb{R},\ \vartheta\in\mathbb{S}^{d-1}, t>0,
\]
$\sigma_1(\d\phi)$ is the surface measure on $\mathbb{S}^{d-1}$ normalised to have unit mass and 
\[
c(\alpha) =2^{\alpha-1}\pi^{-d}\frac{\Gamma((d+\alpha)/2)\Gamma(d/2)}{\big|\Gamma(-\alpha/2)\big|}.
\] Similar calculations as those used in \cite{ALEAKyp} show that the L\'evy system $(H,L)$ of the MAP $(\xi,\Theta)$, associated to the stable process $(X, \mathbb{P})$, see \cite{Cinlar2} for background, is given by the additive functional $H_t:=t,\ t\geq 0$ and the kernel 
\begin{equation}\label{L-kernel}
L_{\vartheta}(\d\phi, \d y):=c(\alpha)\frac{{\rm e}^{yd}}{|{\rm e}^{y}\phi-   \vartheta|^{\alpha+d}}\sigma_1(\d\phi)\d y.\end{equation} So, for any positive predictable process $(G_t,\ t\geq 0),$ and any function $f$ as above, one has
\begin{equation}\label{freejump}
\begin{split}
&\mathbf{E}_{0,\theta}\left(\sum_{s>0} G_sf(s,\xi_{s-},\Delta\xi_{s},\Theta_{s-}, \Theta_s)\right)\\
&=\mathbf{E}_{0,\theta}\left(\int^{\infty}_0\d s G_s \int_{\mathbb{S}^{d-1}}\int_{\mathbb{R}}L_{\Theta_s}(\d\phi, \d y)f(s, \xi_s, y,\Theta_s, \phi)\right).
\end{split}
\end{equation}


We are interested in the characterisation of the L\'evy system of the MAP associated to $(X, \mathbb{P}^\triangleleft)$, via the Lamperti-Kiu transform. To this end, suppose now we write $\mathbf{P}^\triangleleft:=(\mathbf{P}^\triangleleft_{x,\theta}, (x,\theta)\in\mathbb{R}\times\Omega)$ for the probabilities of the MAP that underly $(X,\mathbb{P}^\triangleleft)$.
\begin{proposition}\label{jumprate}
 For any positive predictable process $(G_t,\ t\geq 0),$ and any function $f:(0,\infty)\times\mathbb{R}^2\times\mathbb{S}^{d-1}\times\mathbb{S}^{d-1}\to[0,\infty)$, bounded and measurable, such that $f(\cdot, \cdot, 0,\cdot, \cdot) = 0$, one has
\begin{equation}\label{compensation}
\begin{split}
&\mathbf{E}^\triangleleft_{0,\theta}\left(\sum_{s>0} G_sf(s,\xi_{s-},\Delta\xi_{s},\Theta_{s-}, \Theta_s)\right)\\
&=\mathbf{E}^\triangleleft_{0,\theta}\left(\int^{\infty}_0\d s\, G_s \int_{\mathbb{S}^{d-1}}\int_{\mathbb{R}}L^\triangleleft_{\Theta_s}(\d\phi, \d y)f(s, \xi_s, y,\Theta_s, \phi)\right), \quad \forall \theta\in \Omega,
\end{split}
\end{equation}
where $$L^\triangleleft_{\theta}(\d\phi, \d y):={\rm e}^{\beta y}\frac{M(\phi)}{M(\theta)}L_{\theta}(\d\phi, \d y),\qquad \phi\in\mathbb{S}^{d-1}, y \in \mathbb{R}.$$ That is to say that, under $\mathbf{P}^\triangleleft,$ the instantaneous jump rate when $(\xi_t, \Theta_t) = (x,\vartheta)$ is 
\[
c(\alpha)\frac{{\rm e}^{y(\beta+d)}}{|{\rm e}^{y}\phi-   \theta|^{\alpha+d}}\frac{M(\phi)}{M(\vartheta)}\d y\sigma_1(\d\phi)\d t, \qquad t\geq 0, \theta, \phi\in\Omega
\]
\end{proposition}

%

A better understanding of the MAP that underlies $(X, \mathbb{P}^\triangleleft)$, allows us to deduce e.g. the following result. For $a>0$, define by
\[
\tau^\ominus_{a}=\inf\{t>0: |X_t|> a\}
\]
 and 
 $\overline{m}(\tau^{\ominus}_a -) = \sup\{t<\tau^{\ominus}_a : |X_t| =\sup_{s<t}|X_s| \}$ the last radial maximum before exiting the ball of radius $a$.

\begin{theorem}\label{oo}There exists a probability measure, $\upsilon^*$ on $\Omega$, which is invariant in the sense that 
   \[
\mathbb{P}_{\upsilon^*}^\triangleleft \left(\arg(X_{\tau^\ominus_{{\rm e}}})\in \d \theta\right)   : = \int_{\Omega}\upsilon^*(\d \phi)\mathbb{P}_\phi^\triangleleft \left(\arg(X_{\tau^\ominus_{{\rm e}}})\in \d \theta\right) = \upsilon^*(\d \theta), \qquad \theta\in\Omega,
   \]
such that, for all $x\in\Gamma$, 
under $\mathbb{P}^{\triangleleft}_{r\theta}$, the triple 
\[
\left( \frac{{X}_{\overline{m}(\tau^{\ominus}_a -)}}{a}, \frac{X_{\tau_a^{\ominus}-}}{a},\frac{X_{\tau^{\ominus}_a}}{a}\right)
\]
   converges  in distribution as $a\to\infty$ to a limit which is independent of $r$ and $\theta$ and non-degenerate. Equivalently,  by scaling,  the triple 
   \[
   (X_{\tau_1^{\ominus}},X_{\tau_1^{\ominus}-},X_{m(\tau_1^{\ominus}-)})
   \]  converges in distribution under $\mathbb{P}_{x}$, as $\Gamma\ni x\to0$,  to the same limit.  More precisely, for any  continuous and bounded $f:\Gamma^3\to[0,\infty)$,
   \[
\lim_{\Gamma\ni x\to0}\mathbb{E}^{\triangleleft}_{x}\left[f(X_{\tau_1^{\ominus}},X_{\tau_1^{\ominus}-},X_{m(\tau_1^{\ominus}-)})\right]
= \frac{1}{\mathbb{E}^\triangleleft_{\upsilon^*}[\log|X_{\tau^\ominus_{\rm e}}|]}
\int_{\Omega}  \int_0^{\infty} \upsilon^*(\d \phi){\rm d} r \,G(r,\phi),
\]
where 
\[
G(r,\phi)=\mathbb{E}^{\triangleleft}_{{\rm e}^{-r} \phi}\left[f(X_{\tau_1^{\ominus}},X_{\tau_1^{\ominus}-},X_{m(\tau_1^{\ominus}-)}) \mathbf{1}_{(\tau^\ominus_1\leq \tau_{{\rm e}^{1-r}}^{\ominus})}\right].
\]
\end{theorem}

The above theorem, although seemingly innocent and intuitively clear,  offers us access to a very important result. In order to understand why, we must take a small diversion into radial excursion theory, as described in Kyprianou et al. \cite{Deep3}.

\smallskip

Theorem  \ref{conditioned} shows that  $(X, \mathbb{P}^\triangleleft)$,  is a self-similar Markov process. As mentioned above, it follows that it has a Lamperti-Kiu representation of the form \eqref{eq:lamperti_kiu}, with an underlying MAP, say 
$(\xi, \Theta)$, with probabilities $\mathbf{P}^\triangleleft_{x,\theta}$, $x\in\mathbb{R}$, $\theta\in\mathbb{S}^{d-1}$.
For each $t>0,$ let $\overline{\xi}_t  = \sup_{u\leq t}\xi_u$ and define 
\[
 \texttt{g}_t =\sup\{s<t:\xi_s =\overline{\xi}_t\} \text{ and } \texttt{d}_t = \inf\{s>t:\xi_s =\overline{\xi}_t\},
 \]
 which code the left and right end points of excursions of $\xi$ from its maximum, respectively.
Then, for all $t>0,$ with $\texttt{d}_t>\texttt{g}_t,$ we define the excursion process

\[
(\epsilon_{\texttt{g}_t}(s),\Theta_{\texttt{g}_t}^{\epsilon}(s)):= (\xi_{\texttt{g}_t+s}-\xi_{\texttt{g}_t},\Theta_{\texttt{g}_t+s}) , \qquad s\leq \zeta_{\texttt{g}_t} :=\texttt{d}_t-\texttt{g}_t;
\]
it codes the excursion of $(\overline\xi-\xi, \Theta)$ from the set $(0,\mathbb{S}^{d-1})$ which straddles time $t$. Such excursions live in the space $\mathbb{U}(\mathbb{R}\times\mathbb{S}^{d-1})$, the space of c\`adl\`ag paths 
with lifetime  $\zeta = \inf\{s>0: \epsilon(s) <0\}>0$ such that $(\epsilon(0),\Theta^\epsilon(0))\in \{0\}\times\mathbb{S}^{d-1}$, $(\epsilon(s), \Theta^\epsilon(s))\in (0,\infty)\times\mathbb{S}^{d-1}$, for $0<s<\zeta$,  and $\epsilon(\zeta)\in (-\infty,0]$.

\smallskip

For $t>0,$ let $R_t=\texttt{d}_t-t,$ and define the set  $G = \{t>0: R_{t-}=0, R_{t}>0\} = \{{\texttt g}_s: s\geq 0\}$ of the left extrema of excursions from $0$ for $\overline{\xi}-\xi$.  The classical theory of exit systems in Maisonneuve \cite{M75} now implies that there exist an additive functional $(\ell_t, t\geq 0)$ carried by the set of times $\{t\geq0: (\overline{\xi}_t - {\xi}_t, \Theta_t)\in\{0\}\times\mathbb{S}^{d-1}\}$, with a bounded $1$-potential, and a family of {\it excursion measures}, $(\mathbb{N}^\triangleleft_{\theta }, \theta\in\mathbb{S}^{d-1})$, such that  
\begin{itemize}
\item[(i)] $(\mathbb{N}^\triangleleft_{\theta }, \theta\in\mathbb{S}^{d-1})$ is a kernel from $\mathbb{S}^{d-1}$ to $\mathbb{R}\times\mathbb{S}^{d-1},$ such that $\mathbb{N}^\triangleleft_{\theta}(1-{\rm e}^{-\zeta})<\infty$ and $\mathbb{N}^\triangleleft_{\theta}$ is carried by the set $\{(\epsilon(0),\Theta^\epsilon(0)=(0,\theta)\}$ and $\{\zeta>0\}$ for all $\theta\in\mathbb{S}^{d-1};$
\item[(ii)]we have the {\it exit formula} \begin{align}
&\mathbf{E}^\triangleleft_{x,\theta}\left[\sum_{{\texttt g}\in G}F((\xi_s, \Theta_s): s<{\texttt g})H((\epsilon_{{\texttt g}}, \Theta^\epsilon_{\texttt g}))\right]\notag\\
&\hspace{2cm}=\mathbf{E}^\triangleleft_{x,\theta}\left[\int_0^\infty F((\xi_s, \Theta_s): s< t)\mathbb{N}^\triangleleft_{ \Theta_t}(H(\epsilon, \Theta^\epsilon)){\rm d}\ell_t\right],
\label{exitsystem1}
\end{align}
for $x\neq 0$, where $F$ is continuous on the space of c\`adl\`ag paths $\mathbb{D}(\mathbb{R}\times\mathbb{S}^{d-1})$ and $H$ is measurable on the space of c\`adl\`ag paths $\mathbb{U}(\mathbb{R}\times\mathbb{S}^{d-1});$
\item[(iii)] for any $\theta\in\mathbb{S}^{d-1}$, under the measure $\mathbb{N}^\triangleleft_{\theta},$ the process $((\epsilon(s), \Theta^\epsilon(s)), s<\zeta)$ is Markovian with the same semigroup as $(\xi, \Theta)$ killed at its first hitting time of $(-\infty,0]\times\mathbb{S}^{d-1}.$   
\end{itemize}
The couple $(\ell, (\mathbb{N}^\triangleleft_{\theta}, \theta\in \mathbb{S}^{d-1}))$ is called an exit system. 
In Maisonneuve's original formulation, the pair $\ell$ and the kernel $(\mathbb{N}^\triangleleft_{\theta}, \theta\in\mathbb{S}^{d-1})$ is not unique, but once $\ell$ is chosen, the $(\mathbb{N}^\triangleleft_{\theta}, \theta\in\mathbb{S}^{d-1})$ is determined but for a $\ell$-neglectable set, i.e. a set $\mathcal{A}$ such that 
\[
\mathbf{E}^\triangleleft_{x,\theta}\left[\int_{t\geq 0}\mathbf{1}_{((\overline{\xi}_s-{\xi}_s,\Theta_s)\in\mathcal{A})}{\rm d}\ell_s\right]=0.
\] 

\smallskip

Let $(\ell^{-1}_t, t\geq 0)$ denote the right continuous inverse of $\ell,$ $H^+_t :=\xi_{\ell^{-1}_t}$ and  $\Theta^+_t = \Theta_{\ell^{-1}_t}$, $t\geq 0$. The strong Markov property tells us that $(\ell^{-1}_t, H^+_t, \Theta^+_t)$, $t\geq 0$, defines a Markov additive process, whose first two elements are ordinates that are non-decreasing.  Rotational invariance of $X$ implies that $\xi$, alone, is also a L\'evy process, then the pair $(\ell^{-1}, H^+)$, without reference to the associated modulator $\Theta^+$, are Markovian and play the role of the ascending ladder time and height subordinators of $\xi$. But here, we are more concerned  with their dependency on $\Theta^+$.

\smallskip

Taking account of the Lamperti--Kiu transform \eqref{eq:lamperti_kiu}, it is natural to consider how the excursion of $(\overline{\xi} - {\xi}, \Theta)$ from $\{0\}\times\mathbb{S}^{d-1}$ translates into a radial excursion theory for the process 
\[
Y_t  : = {\rm e}^{\xi_t}\Theta_t, \qquad t\geq 0.
\]
Ignoring the time change in \eqref{eq:lamperti_kiu}, we see that the radial maxima of the process $Y$ agree with the radial maxima of the stable process $X$.
Indeed, an excursion of $(\overline{\xi} - {\xi}, \Theta)$ from $\{0\}\times\mathbb{S}^{d-1}$ constitutes an excursion of $(Y_t/\sup_{s\leq t}|Y_s|, t\geq 0)$, from $\mathbb{S}^{d-1}$, or equivalently an excursion of $Y$ from its running radial supremum. Moreover, we see that, for  all $t>0$ such that ${\texttt d}_t > {\texttt g}_t$, 
\[
Y_{\texttt{g}_t + s} = {\rm e}^{\xi_{\texttt{g}_t}} {\rm e}^{\epsilon_{\texttt{g}_t}(s)} \Theta^\epsilon_{{\texttt g}_t}(s) = |Y_{\texttt{g}_t}|{\rm e}^{\epsilon_{\texttt{g}_t}(s)} \Theta^\epsilon_{{\texttt g}_t}(s)=:  |Y_{\texttt{g}_t}|\mathcal{E}_{\texttt{g}_t}(s)
, \qquad s\leq \zeta_{{\texttt g}_t}.
\]

Whilst a cluster of papers on the general theory of Markov additive processes exists in the literature from the 1970s and 1980s, see e.g. \c{C}inlar \cite{Cinlar, Cinlar1, Cinlar2} and Kaspi \cite{Kaspi}, as well as in the setting that $\Theta$ is a discrete process, see Asmussen \cite{AsmussenQueue} and Albrecher and Asmussen \cite{AA}, as well as some recent advances, see the Appendix in Dereich at al. \cite{DDK}, relatively little is known about the fluctuations of MAPs in comparison to e.g. L\'evy processes. Note the latter are a degenerate class of MAPs, in the sense that a L\'evy process can be seen as MAP with constant driving process.   

\smallskip

A good example of an open problem pertaining to the fluctuation theory of MAPs is touched upon in Theorem \ref{oo}: Suppose that $\Theta$ has a stationary distribution, under what conditions does $\Theta^+$ have a stationary distribution? This is a question that has been raised in general in Section 4 of the paper~\cite{kaspi-excursions}. Below we give a complete answer in the present setting.

\begin{theorem}\label{ooC}
Under $\mathbf{P}^\triangleleft$, the modulator process $\Theta^{+}$ has a stationary distribution, that is 
 \[
  \pi^{\triangleleft,+}(\d\theta) := \lim_{t\to\infty}\mathbf{P}^\triangleleft_{x,\theta}(\Theta^+_t\in \d \theta), \qquad \theta\in\Omega, x\in\mathbb{R},
  \]
  exists as a non-degenerate distributional weak limit. 
\end{theorem}
\smallskip

\begin{remark}\label{pigamma}\rm
The reader will also note that, thanks to the change of measure 
 \eqref{otherCOM} combined with  the fact that $\ell^{-1}_t$ is an almost surely finite stopping time and Theorem III.3.4  of \cite{JS} , the ascending ladder MAP process $(H^+,\Theta^+),$ under $\mathbf{P}^\triangleleft$, has the property 
\begin{equation}
\label{plusCOM}
\left.\frac{\d \mathbf{P}^\triangleleft_{0,\theta}}{\d \mathbf{P}_{0,\theta}}\right|_{\sigma((H_s^+, \Theta^+_s), s\leq t)} = 
{\rm e}^{\beta H^+_t}\frac{M(\Theta^+_t)}{M(\theta)}\mathbf{1}_{(t<
\texttt{k}^{\Omega, +}
)}, \qquad t\geq 0,
\end{equation}
where $\texttt{k}^{\Omega,+} = \inf\{t>0: \Theta_t^+ \not\in\Omega\}$
As a consequence of the existence of $\pi^{\triangleleft,+}$, we note that 
\[
\pi^{\Gamma,+}(\d \theta) := \frac{1}{M(\theta)}\pi^{\triangleleft,+}(\d\theta), \qquad \theta\in\Omega,
\]
is an invariant distribution for $(\Theta^+_t\mathbf{1}_{(t<
\texttt{k}^{\Omega, +})}, t\geq 0)$ under $\mathbf{P}$.
\end{remark}

\section{Auxiliary results for dual processes in the cone}

In order to prove some of the results listed above, 
we will need to understand another type of conditioned process, namely the stable process conditioned to continuously absorb at the origin.

\begin{theorem}\label{conditionapex}
{\color{black} For $A\in \mathcal{F}_t$, on the space of c\`adl\`ag paths in $\Gamma$ with a cemetery state at the apex of $\Gamma$, 
$$ \mathbb{P}^{\triangleright}_x(A, \, t<\kappa_{\{0\}} ): =\lim_{a\to0} \mathbb{P}^\triangleleft_x(A,\, t< \tau^\oplus_a|\tau^\oplus_a <\infty), $$
 is well defined as a stochastic process which is continuously absorbed at the apex of $\Gamma$, where $\kappa_{\{0\}} = \inf\{t>0: |X_t| =0\}$ and $\tau^\oplus_a =\inf\{s>0:|X_s|<a\}$. 
Moreover, for $A\in\mathcal{F}_t$, 
\begin{equation}
\mathbb{P}^{\triangleright}_x(A, \, t<\kappa_{\{0\}})
=\mathbb{E}_x\left[
\mathbf{1}_{(A, \, t<\kappa_\Gamma)}\frac{H(X_t)}{H(x)}\right], \qquad t\geq 0,
\label{COMH}
\end{equation}
where 
\[
H(x) = 
|x|^{\alpha-\beta-d}M(\arg(x)).
\]
}
Next, we write $\mathbf{P}^\triangleright:=(\mathbf{P}^\triangleright_{x,\theta}, x\in\mathbb{R}, \theta\in\Omega)$ for the probability law of the MAP that underlies $(X,\mathbb{P}^\triangleright)$. For any positive predictable process $(G_t,\ t\geq 0),$ and any function $f:(0,\infty)\times\mathbb{R}^2\times\mathbb{S}^{d-1}\times\mathbb{S}^{d-1}\to\mathbb{R}$, bounded and measurable, such that $f(\cdot, \cdot, 0,\cdot, \cdot) = 0$, one has
\begin{equation}
\begin{split}
&\mathbf{E}^\triangleright_{0,\theta}\left(\sum_{s>0} G_sf(s,\xi_{s-},\Delta\xi_{s},\Theta_{s-}, \Theta_s)\right)\\
&=\mathbf{E}^\triangleright_{0,\theta}\left(\int^{\infty}_0\d s G_s \int_{\mathbb{S}^{d-1}}\int_{\mathbb{R}}L^\triangleright_{\Theta_s}(\d\phi, \d y)f(s, \xi_s, y,\Theta_s, \phi)\right), \quad \forall \theta\in \mathbb{S}^{d-1},
\end{split}
\end{equation}
where $$L^\triangleright_{ \theta}(\d\phi, \d y):={\rm e}^{\beta y}\frac{H(\phi)}{H(\theta)}L_{\theta}(\d\phi, \d y),\qquad \phi\in\mathbb{S}^{d-1}, y \in \mathbb{R}.$$ That is to say, that under $\mathbf{P}^\triangleright,$  the instantaneous jump rate when $(\xi_t, \Theta_t) = (x,\vartheta)$ is 
\[
c(\alpha)\frac{{\rm e}^{(\beta + d) y}}{|{\rm e}^{y}\phi-   \theta|^{\alpha+d}}\frac{H(\phi)}{H(\vartheta)}\d y\sigma_1(\d\phi)\d t, \qquad t> 0, \theta, \phi\in\Omega
\]

\end{theorem}

As alluded to, the process $(X, \mathbb{P}^{\triangleright})$,  is intimately related to the process $(X, \mathbb{P}^\triangleleft)$. This is made clear in our final main result which has the flavour of the Riesz--Bogdan--\.Zak transform; cf. Bogdan and \.Zak \cite{BZ}. For the sake of reflection it is worth stating the  Riesz--Bogdan--\.Zak transform immediately below first, recalling the definition of $L$-time and then our first main result in this section.

\begin{theorem}[Riesz--Bogdan--\.Zak transform]\label{RBSthrm}
Suppose we write $Kx = x/|x|^2$, $x\in\mathbb{R}^d$ for the classical inversion of space through the sphere $\mathbb{S}^{d-1}$. Then, in dimension $d\geq 2$,   for $x\neq 0$, $(KX_{\eta(t)}, t\geq 0)$ under $
\mathbb{P}_{Kx}$ is equal in law to $(X_t, t\geq 0)$ under $\mathbb{P}^\circ_{x}$, where
\begin{equation}
\left.\frac{\d \mathbb{P}^\circ_x}{\d \mathbb{P}_x}\right|_{\sigma(X_s: s\leq t)} = \frac{|X_t|^{\alpha-d}}{|x|^{\alpha-d}}, \qquad t\geq 0
\label{COM}
\end{equation}
and $\eta(t) = \inf\{s>0: \int_0^s|X_u|^{-2\alpha}\d u>t\}$. 
\end{theorem}

Hereafter, by an $L$-time we mean the following. Suppose that $\mathcal{G}$ is the sigma-algebra generated by $X$ and  write $\mathcal{G}(\mathbb{P}^\triangleleft_\nu)$ for its completion by the null sets of $\mathbb{P}^\triangleleft_\nu$, where $\nu$ is a randomised initial distribution. Moreover, write $\overline{\mathcal G} =\bigcap_{\nu} \mathcal{G}(\mathbb{P}^\triangleleft_\nu)$, where the intersection is taken over all probability measures on the state space of $X$.	 A finite  random time $\texttt{k}$ is called an $L$-time (generalized last exit time) if
$\{s<\texttt{k}(\omega)-t\}=\{s<\texttt{k}(\omega_t)\}$ for all $t,s\geq 0$.
(Normally, we must include in the definition of an $L$-time that $\texttt{k}\leq \zeta$,
where $\zeta$ is the first entry of the process to a cemetery state. However, this is not applicable for $(X, \mathbb{P}^\triangleleft)$.)
 The two most important examples of $L$-times are killing times and last  exit times. 

\begin{theorem}\label{BZcone}
Consider again the transformation of space via the sphere inversion $Kx = x/|x|^2$, $x\in\mathbb{R}^d$. 
\begin{itemize}
\item[(i)] The process $(KX_{\eta(t)}, t\geq 0)$ under $\mathbb{P}^\triangleleft_x$, $x\in\Gamma$, is equal in law to $(X_t, t<\kappa_{\{0\}})$ under $\mathbb{P}^\triangleright_x$, $x\in\Gamma$, where
\begin{equation}
\eta(t) =\inf\{s>0: \int_0^s |X_u|^{-2\alpha}\d u>t\}, \qquad t\geq 0.
\label{etatimechange}
\end{equation}
and $\kappa_{\{0\}} = \inf\{t>0 : X_t =0\}$.
\item[(ii)] Under $\mathbb{P}_0^\triangleleft$, the time reversed process 
\[
{\stackrel{_\leftarrow}{X}}_t : = X_{(\emph{\texttt{k}}-t)-}, \qquad t\leq \emph{\texttt{k}},
\]
is a homogenous strong Markov process whose transitions agree with those of $(X, \mathbb{P}^\triangleright_x)$, $x\in\Gamma$, where $\emph{\texttt{k}}$ is 
an $L$-time of $(X, \mathbb{P}^\triangleleft_x)$, $x\in\Gamma\cup\{0\}$.
\end{itemize}
\end{theorem}

\medskip

Our third main theorem considers the possibility of a recurrent extension from the origin of of $(X, \mathbb{P}^\triangleright)$, similar in spirit to Theorem \ref{recex}.

\begin{theorem}\label{recexctsat0}
Let $\mathbf{N}^{\triangleright}$ be a self-similar excursion measure compatible with $(X,\mathbb{P}^{\triangleright}).$ We have that there exists a $\gamma\in (0, {\alpha^{-1}}(d+2\beta-\alpha)\wedge 1)$, a constant $a\geq 0$, and a measure $\pi^{\triangleright}$ on $\Omega$ such that $a\pi^{\triangleright}\equiv 0,$ $\int_{\Omega}\pi^{\triangleright}(\d\theta)H(\theta)<\infty,$ and  $\mathbf{N}^{\triangleright}$ can be represented by, for any $t>0,$ and any $A\in\mathcal{F}_t$ 
\begin{equation}\label{EL-decomp3}
\begin{split}
\mathbf{N}^{\triangleright}(A \,,\,  t<\zeta) =a\mathbb{E}^{\triangleleft}_0\left[\frac{H(X_t)}{M(X_t)}\mathbf{1}_A\right] +\int_0^\infty\frac{\d r}{r^{1+\alpha\gamma}}\int_{\Omega}\pi^{\triangleright}(\d\theta)\,\mathbb{E}^{\triangleright}_{r\theta}[A \, ,\, t<\kappa_\Gamma].
\end{split}
\end{equation}  
If $0<\beta<(2\alpha-{d})/{2},$ and $a>0,$ the process $(X, \mathbb{P}^{\triangleright})$ has unique recurrent extension that leaves $0$ continuously. If $\beta\geq (2\alpha-{d})/{2},$ then $a=0,$ and there is no recurrent extension that leaves $0$ continuously.  
\smallskip

Conversely, for each $\gamma\in (0, \alpha^{-1}(d+2\beta-\alpha)\wedge 1)$, and $\pi^{\triangleright}$ a non-trivial measure satisfying the above conditions,  there is a unique recurrent extension that leaves zero by a jump and such that 
$$\mathbf{N}^{\triangleright}\left(|X_{0+}|\in \d r, \arg(X_{0+})\in \d \theta\right)=\frac{\d r}{r^{1+\alpha\gamma}}\pi^{\triangleright}(\d \theta), \qquad r>0, \theta\in \Omega.$$
\smallskip

Finally, any self-similar recurrent extension of $(X,\mathbb{P}^{\triangleright})$ with excursion measure $\mathbf{N}^{\triangleright},$ has an invariant measure 
\begin{equation*}
\begin{split}
\cenpi\!{}^\triangleright (\d x) &:=\mathbf{N}^{\triangleright}\left(\int^{\zeta}_{0}\mathbf{1}_{(X_t\in \d x)}\d t\right)
\\
&= a|x|^{2({\alpha -d-\beta})}M(x)^2\d x+\int_0^\infty\frac{\d r}{r^{1+\alpha\gamma}}\int_{\Omega}\pi^{\triangleright}(\d\theta)\,\mathbb{E}^{\triangleright}_{r\theta}\left[\int^{\kappa_\Gamma}_{0}\mathbf{1}_{(X_t\in \d x)}\d t\right],
\end{split}
\end{equation*}
which is unique up to a multiplicative constant, and this measure is sigma-finite but not finite. 
\end{theorem}
It is interesting to remark here that if the cone is such that $\beta\geq (2\alpha-d)/{2},$ or equivalently $(d+2\beta-\alpha)/\alpha\geq 1,$ there is no recurrent extension that leaves zero continuously. This is due to the fact that the closer $\beta$ is to $\alpha$ the smaller the cone is. Because the process is conditioned to hit zero continuously, a process starting from zero should return too quickly to zero, forcing there to be many small excursions,  whose lengths become increasingly difficult to glue end to end in any finite interval of time. We could understand this phenomena with heuristic language by saying that `the conditioned stable process is unable to escape the gravitational attraction to the origin because of  the lack of space needed to do so'.

\medskip

The rest of this paper is organised as follows. In the next section we give the proof of Theorem \ref{conditioned}.
 Thereafter, we prove the above stated results in an order which differs from their presentation. We prove Proposition \ref{jumprate} in Section \ref{provejumprate} and then turn to the proof of Theorem \ref{oo} and Corollary \ref{ooC} in Sections \ref{proveoo} and \ref{proveooC}, respectively.
This gives us what we need to construct the process conditioned to continuously absorb at the apex of $\Gamma$, i.e. Theorem \ref{conditionapex},  in Section \ref{proofconditionapex}. With all these tools in hand, we can establish the duality properties of Theorem \ref{BZcone} in Section \ref{proofBZcone}. Duality in hand, in Section \ref{proof0}, we can return to the Skorokhod convergence of the conditioned process $(X, \mathbb{P}^\triangleleft_x)$, $x\in\Gamma$, to the candidate for $(X,\mathbb{P}^\triangleleft_0)$, described in \eqref{P0}, and prove Theorem \ref{0}. Finally, in Sections \ref{proofrecex} and \ref{proofrecexctsat0}, we complete the paper by looking at the recurrent extension of the conditioned processes in Theorem \ref{recex} and \ref{recexctsat0} respectively.

\section{Proof of Theorem \ref{conditioned}}
We break the proof into the constituent parts of the statement of the theorem.
\subsection{Proof of part (i)}  
For $A\in\mathcal{F}_t$ and $0\neq x\in\Gamma$,
\begin{equation}
\mathbb{P}^{\triangleleft}_x(A,\, t<\zeta)=\lim_{s\to\infty}\mathbb{E}_x\left[  \mathbf{1}_{(A\cap \{t <\kappa_\Gamma\})}\frac{\mathbb{P}_{X_t}(\kappa_\Gamma>s)}{\mathbb{P}_{x}(\kappa_\Gamma>t+s)}  \right].
\label{conditionallimit}
\end{equation}
From Lemma 4.2 of \cite{BB}, for $t^{-1/\alpha}|x|<1$, we have the bound 
$$\frac{\mathbb{P}_x(\kappa_{\Gamma}>t)}{M(x)t^{-\beta/\alpha}}\in [C^{-1},C], $$
for some $C>1$.
Otherwise, if $t^{-1/\alpha}|x|>1$ then,
$\mathbb{P}_x(\kappa_{\Gamma}>t)\leq 1 <|x|^\beta t^{-\beta/\alpha}$. Hence, noting that $M$ is uniformly bounded from above, we have that, for all $x\in\Gamma$ and $s>0$, there is a constant $C'$ such that  $\mathbb{P}_x(\kappa_{\Gamma}>s)\leq C'|x|^\beta s^{-\beta/\alpha}$.
Hence, for $s$ sufficiently large, there is another constant $C''$ (which depends on $x$) such that
\[
\frac{\mathbb{P}_{X_t}(\kappa_\Gamma>s)}{\mathbb{P}_{x}(\kappa_\Gamma>t+s)}\leq \frac{C'|X_t|^{\beta}s^{-\beta/\alpha}}{C^{-1}M(x)(t+s)^{-\beta/\alpha}}<C''|X_t|^{\beta}.
\]
It is well known that $X_t$ has all absolute moments of any order in $(0,\alpha)$; cf. Section 25 of Sato \cite{Sato}. The identity \eqref{MCOM}  now follows from Proposition \ref{BPWprop} and the Dominated Convergence Theorem. Furthermore, by construction, for any $x\in \Gamma,$ $$\mathbb{P}^{\triangleleft}_x(t<\texttt{k}) =1, \qquad \forall t\geq 0.$$ It thus follows that under $\mathbb{P}^\triangleleft,$ $X$ has an infinite lifetime.

%
%

\smallskip

\subsection{Proof of part (ii)}\label{part2} That $(X, \mathbb{P}^\triangleleft)$ is a ssMp is a consequence of $(X, \mathbb{P})$ having the scaling property and the strong Markov property. Indeed, $(X, \mathbb{P}^\triangleleft)$ is a strong Markov process, since it is obtained via an h-transform of $(X,\mathbb{P}).$ To verify that it has the scaling property,  let $c>0$ and define $\tilde{X}_t: =  cX_{c^{-\alpha }t}$, $t\geq 0$. We have that 
\begin{equation}
\tilde{\kappa}_\Gamma:=\inf\{t>0: \tilde{X}_t\notin \Gamma\}= c^\alpha \kappa_\Gamma,
\label{stopscale}
\end{equation}
 and by the scaling property 
\begin{equation}\label{scaling}
(\tilde{X},\mathbb{P}_x)\stackrel{\text{Law}}{=}(X,\mathbb{P}_{cx}),\qquad x\in \Gamma.
\end{equation}  Considering the transition probabilities of $(X, \mathbb{P}^\triangleleft)$, we note with the help of \eqref{scaling} and \eqref{1/a} that, for bounded and measurable $f$, 
\begin{align*}
\mathbb{E}_x^\triangleleft [f (\tilde{X}_{ t}) ]& = \mathbb{E}_x\left[\mathbf{1}_{(c^{-\alpha}t<\kappa_\Gamma
)}f(c X_{c^{-\alpha }t})\frac{M( X_{c^{-\alpha } t}   )}{M(x)}\right]\\
&=\mathbb{E}_x\left[\mathbf{1}_{(t<\tilde\kappa_\Gamma
)}f(\tilde{X}_t)\frac{|\tilde{X}_t|^\beta M( \tilde{X}_{t}/|\tilde{X}_{t}|   )}{|cx|^\beta M(cx/ |cx|)}\right]\\
&=\mathbb{E}_{cx}\left[\mathbf{1}_{(t<\kappa_\Gamma
)}f({X}_t)\frac{|{X}_t|^\beta M( {X}_{t}/|{X}_{t}|   )}{|cx|^\beta M(cx/ |cx|)}\right]\\
&=\mathbb{E}_{cx}^\triangleleft [f ({X}_{ t}) ]
, \qquad x\in\Gamma.
\end{align*}
This last observation together with the Markov property ensures the required self-similarity of $(X, \mathbb{P}^\triangleleft)$. \hfill$\square$

%
%

\section{Proof of Proposition \ref{jumprate}}\label{provejumprate} We use a method taken from  Theorem I.3.14 of \cite{ALEAKyp}. 
From the  Lamperti--Kiu transformation \eqref{eq:lamperti_kiu}, we have 
\begin{equation}
\label{invL}
\xi_{t}=\log(|X_{{A}(t)}|/|X_{0}|),\qquad \Theta_{t}=\frac{X_{{A}(t)}}{|X_{{A}(t)}|},\qquad t\geq 0,
\end{equation}
where ${A}(t) = \inf\{s>0: \int^{s}_{0}|X_{u}|^{-\alpha}\d u> t\}$.

\smallskip

To show that \eqref{compensation} holds,
we first note that,  on account of the fact that $A(t)$ in \eqref{invL} is an almost surely finite stopping time {\color{black} and the simple relation $A({\texttt{k}}^\Omega) = \kappa_\Gamma$, where ${\texttt{k}}^\Omega = \inf\{t>0: \Theta_t \not\in \Omega\}$}, from the martingale property in \eqref{MCOM}, we have by Theorem III.3.4  of \cite{JS} that 
\begin{equation}
\left.\frac{\d \mathbf{P}^\triangleleft_{x,\theta}}{\d \mathbf{P}_{x,\theta}}\right|_{\mathcal{G}_t} = 
{\rm e}^{\beta (\xi_t-x)}\frac{M(\Theta_t)}{M(\theta)}\mathbf{1}_{(t<
{\texttt{k}}^\Omega
) 
},
\qquad t\geq 0,
\label{otherCOM}
\end{equation}
where   $\mathcal{G}_t  =\sigma((\xi_s, \Theta_s), s\leq t)$, $t\geq 0$.

\smallskip

Now write
\[
\mathbf{E}^\triangleleft_{0,\theta}\left(\sum_{s>0} G_sf(s,\xi_{s-},\Delta\xi_{s},\Theta_{s-}, \Theta_s)\right) 
= \lim_{t\to\infty}\mathbf{E}_{0,\theta}\left(\mathcal{M}_t
\sum_{0<s\leq t} G_sf(s,\xi_{s-},\Delta\xi_{s},\Theta_{s-}, \Theta_s)\right)
\]
where 
$\mathcal{M}_t  = \mathbf{1}_{(t<\texttt{k}^\Omega)}{\rm e}^{\beta\xi_t}M(\Theta_t)/M(\theta)$, $t\geq 0$ is the martingale density from the change of measure \eqref{MCOM}.
Suppose we write $\Sigma_t$ for the sum term in the final expectation above. The semi-martingale change of variable formula tells us that 
\[
 \mathcal{M}_t\Sigma_t = \mathcal{M}_0(\theta)\Sigma_0 + \int_0^t \Sigma_{s-}\d \mathcal{M}_s + \int_0^t\mathcal{M}_{s-}\d \Sigma_s +  [\mathcal{M}, \Sigma]_t, \qquad t\geq0,
\]
where $[\mathcal{M}, \Sigma]_t$ is the quadratic co-variation term. 
On account of the fact that  $(\Sigma_t, t\geq0)$, has bounded variation, the latter term takes the form $[\mathcal{M}, \Sigma]_t = \sum_{s\leq t}\Delta M_t \Delta \Sigma_ t$. As a consequence 
\begin{equation}
 \mathcal{M}_t\Sigma_t = \mathcal{M}_0(\theta)\Sigma_0 + \int_0^t \Sigma_{s-}\d \mathcal{M}_s + \int_0^t\mathcal{M}_{s}\d \Sigma_s , \qquad t\geq0,
\label{smg}
\end{equation}Moreover, after taking expectations, as the first in integral in \eqref{smg} is a martingale and $\Sigma_0 =0$, the only surviving terms give us 
 \begin{align*}
 \mathbf{E}^\triangleleft_{0,\theta}&\left(\sum_{s>0} G_sf(s,\xi_{s-},\Delta\xi_{s},\Theta_{s-}, \Theta_s)\right) \\
 &=\mathbf{E}_{0,\theta}\left(
\sum_{s>0}\mathbf{1}_{(t<\texttt{k}^\Gamma)}{\rm e}^{\beta\xi_s}\frac{M(\Theta_s)}{M(\theta)}G_sf(s,\xi_{s-},\Delta\xi_{s},\Theta_{s-}, \Theta_s)\right)\\
&=\mathbf{E}_{0,\theta}\left[\int^{\infty}_{0}\d s\,  G_s\mathbf{1}_{(s<\texttt{k}^\Gamma)} {\rm e}^{\beta\xi_s}\frac{M(\Theta_s)}{M(\theta)} \int_{\Omega}\sigma_1(\d\phi)\int_{\mathbb{R}}\d y \frac{c(\alpha){\rm e}^{y(\beta + d)}}{|{\rm e}^{y}\phi-     \Theta_{s}|^{\alpha+d}}
\frac{M(\phi)}{M(\Theta_s)}f(s, \xi_{s}, y, \Theta_{s}, \phi)\right]\notag\\
&=\mathbf{E}^{\triangleleft}_{0,\theta}\left[\int^{\infty}_{0}\d s\,  G_s \int_{\Omega}\sigma_1(\d\phi)\int_{\mathbb{R}}\d y \frac{c(\alpha){\rm e}^{y(\beta + d)}}{|{\rm e}^{y}\phi-     \Theta_{s}|^{\alpha+d}}
\frac{M(\phi)}{M(\Theta_s)}f(s, \xi_{s}, y, \Theta_{s}, \phi)\right]\notag\\
 \end{align*}
 where in the second equality we have used the jump rate \eqref{freejump} and in the third Fubini's theorem together with \eqref{otherCOM}.
\hfill$\square$

\section{Proof of Theorem \ref{oo}}\label{proveoo}
At the root of our proof of  Theorem \ref{oo}, we will appeal to Markov additive renewal theory in the spirit of Alsmeyer \cite{Alsmeyer1994, Alsmeyer2014},  Kesten \cite{kestenMAP} and Lalley \cite{lalley}.
The radial excursion theory we have outlined in Section \ref{radial}
is a natural mathematical pre-cursor to Markov additive renewal theory, however, one of the problems we have at this point in our reasoning, as it will be seen later, is that it is not yet clear whether there is a stationary behaviour for the process $(\Theta^+,\mathbf{P}^\triangleleft)$, of the radial ascending ladder MAP. Indeed, as already discussed, we will deduce from our calculations here that a stationary distribution does indeed exist in Corollary \ref{ooC}.

\smallskip

We build instead an alternative Markov additive renewal theory around a naturally  chosen discrete subset set ladder points which behave well into the hands of the scaling property of $(X, \mathbb{P}^\triangleleft)$. As the proof is quite long, we break the remainder of this section into a number of steps, marked by subsections. The proof of Theorem \ref{oo} will thus be our.

 \subsection{A discrete ladder MAP}  Under $\mathbb{P}^{\triangleleft}$, define the following sequence of stopping times,
 \[
T_n :=\inf\{t>T_{n-1} : |X_t|> {\rm e}|X_{T_{n-1}}|\}, \qquad n\geq1,
\]
with $T_0=0$, and   
\[
S_n = \sum_{k = 1}^n  A_k\qquad 
 A_n   = \log\frac{|X_{T_n}|}{|X_{T_{n-1}}|}\quad  \text{ and }\quad \Xi_n=\arg(X_{T_n}), \qquad n\geq 1.
 \]
Note in particular that 
\[
X_{T_n}=|x|{\rm e}^{S_n}\Xi_n, \qquad n\geq 1.
\]

Then,  we claim that  $((S_n,\Xi_n), n\geq 0)$, is a Markov additive renewal process. 
To verify this claim, we appeal principally to the strong  Markov and scaling property  of $(X, \mathbb{P}^{\triangleleft})$. Indeed, for any $x\in \Gamma$, we have that for any finite stopping time $T$, under $\mathbb{P}^{\triangleleft}_x,$ the conditional law of $(X_{T+s}, s>0)$ given $(X_u, u\leq T)$ equals that of $(|y|X_{s/|y|^\alpha}, s>0)$ under $\mathbb{P}^\triangleleft_{\arg(y)},$ with $y=X_T.$ Hence, for any $n\geq 0$, by construction, conditionally on $(X_u, u\leq T_n)$ we have that
\begin{equation*}
\begin{split}
T_{n+1} &=\inf\{t>T_{n} : |X_t|> {\rm e}|X_{T_{n}}|\}\\
&=T_n+\inf\{s>0 : |X_{s+T_n}|> {\rm e}|X_{T_{n}}|\}\\
&\stackrel{\text{Law}}{=}T_n+\inf\{s>0 : |X_{T_n}||\tilde{X}_{s/|X_{T_n}|^{\alpha}}|> {\rm e}|X_{T_{n}}|\}\\
&=T_n+|X_{T_{n}}|^{\alpha}\tilde{T}_1,
\end{split}
\end{equation*}
where $\tilde{X}$ depends on $(X_u, u\leq T_n)$ only through $\arg(X_{T_n}),$ has the same law as $(X,\mathbb{P}^{\triangleleft}_{\arg(X_{T_n})}),$ and $\tilde{T_1}=\inf\{t>0 : |\tilde{X}_t|> {\rm e}\}.$ From these facts, it follows that for any bounded and measurable $f$ on $\mathbb{R}\times\Omega$,
\begin{align*}
&\mathbb{E}^{\triangleleft}_x\left[f(S_{n+1}-S_n,\Xi_{n+1})|(S_i,\Xi_i) : i\leq n\right]\\
&=\mathbb{E}^{\triangleleft}_{x}\left[f\left(\log\frac{|X_{T_{n+1}}|}{|X_{T_n}| },\arg(X_{T_{n+1}})\right)|X_{T_{i}}: i\leq n\right]\\
&=\mathbb{E}^{\triangleleft}_{y}\left[f(\log|X_{T_1}|,\arg(X_{T_1}))\right]|_{y = \Xi_n}.
\end{align*}These calculations ensure that $((S_n,\Xi_n), n\geq 0)$, is a Markov additive renewal process. 
Note, this computation also shows that, under $\mathbb{P}^{\triangleleft},$ the modulator $\Xi : = (\Xi_n, n\geq0)$ is also a Markov process. 

\subsection{Application of Markov additive renewal theory}
Let us introduce the MAP renewal function associated to $(S,\Xi)$, for $\Omega,$
\begin{equation}
 V_{\theta}({\rm d} r,{\rm d} \phi) :=\sum_{n=0}^{\infty}\mathbb{P}^{\triangleleft}_{\theta}(S_n \in{\rm d}r,\Xi_n \in{\rm d}\phi), 
 \qquad r\in\mathbb{R}, \phi\in\Omega.
 \qquad 
 \label{Vrenewal}
 \end{equation}
 We will next show that the joint law in Theorem \ref{oo} can be expressed in term of a renewal like equation involving $V_{\theta}.$ 
\begin{lemma}\label{firstRT} For measurable  $f:\Gamma^3\to [0,\infty)$, we have, for $x\in\Gamma\cap B_1$,
\begin{equation}
 \mathbb{E}^{\triangleleft}_x\left[f(X_{\tau_1^{\ominus}},X_{\tau_1^{\ominus}-},X_{m(\tau_1^{\ominus}-)})\right]=\int_0^{-\log|x|} \int_{ \Omega} V_{\arg(x)}({\rm d} r, {\rm d} \phi) G(-\log|x|-r,\phi),
\label{applyMAPRT}
\end{equation}
where, for $\phi\in\Omega$ and $y\geq 0$,
\begin{align}
G(y,\phi)&:=\mathbb{E}^{\triangleleft}_{{\rm e}^{-y} \phi}\left[f(X_{\tau_1^{\ominus}},X_{\tau_1^{\ominus}-},X_{m(\tau_1^{\ominus}-)})\mathbf{1}_{(\tau^\ominus_1\leq \tau_{{\rm e}^{1-y}}^{\ominus})}
\right]
\label{Gdef}
\end{align}
\end{lemma}
\begin{proof}
Noting that $|X_{T_n}|=|x|{\rm e}^{S_n}$ and $\arg(X_{T_n})=\Xi_n$. Appealing to the strong Markov property we get
\begin{align*}&\mathbb{E}^{\triangleleft}_x\left[f(X_{\tau_1^{\ominus}},X_{\tau_1^{\ominus}-},X_{m(\tau_1^{\ominus}-)})\mathbf{1}_{(\tau_1^{\ominus}<\infty)}
\right]\\
&=\mathbb{E}^{\triangleleft}_x\left[ \sum_{n\geq 0} \mathbf{1}_{(T_n< \tau_1^{\ominus}\leq  T_{n+1})}f(X_{\tau_1^{\ominus}},X_{\tau_1^{\ominus}-},X_{m(\tau_1^{\ominus}-)})\right]\\
&=\mathbb{E}^{\triangleleft}_x\left[ \sum_{n\geq 0} \mathbf{1}_{(T_n< \tau_1^{\ominus})}
\mathbb{E}_{y}\left[f(X_{\tau_1^{\ominus}},X_{\tau_1^{\ominus}-},X_{m(\tau_1^{\ominus}-)})\mathbf{1}_{(\tau_1^{\ominus}\leq T_1)}\right]_{y=X_{T_n}}\right]\\
&=\mathbb{E}^{\triangleleft}_{x}\left[ \sum_{n=0}^{\infty}\mathbf{1}_{(|x|{\rm e}^{S_n} < 1)}\mathbb{E}_{y}\left[f(X_{\tau_1^{\ominus}},X_{\tau_1^{\ominus}-},X_{m(\tau_1^{\ominus}-)})\mathbf{1}_{(\tau_1^{\ominus}\leq  T_1)}\right]_{y=|x|{\rm e}^{S_n}\Xi_n}\right],
\end{align*}
where in the first equality the indicator implies $T_n \leq m(\tau_1^{\ominus}-)$.
We can thus write 
\begin{align*}
&\mathbb{E}^{\triangleleft}_x\left[f(X_{\tau_1^{\ominus}},X_{\tau_1^{\ominus}-},X_{m(\tau_1^{\ominus}-)})\right]\notag\\
&=\int_{0}^{-\log|x|} \int_{\Omega} V_{\arg(x)}({\rm d} r, {\rm d} \phi)\mathbb{E}^{\triangleleft}_{|x|{\rm e}^r \phi}\left[f(X_{\tau_1^{\ominus}},X_{\tau_1^{\ominus}-},X_{m(\tau_1^{\ominus}-)})\mathbf{1}_{(\tau_1^{\ominus}\leq  \tau_{|x|{\rm e}^{r+1}}^{\ominus})}\right],
\end{align*}
which agrees with the statement of the lemma.
\end{proof}
We now have all the elements to explain the strategy we will follow to prove Theorem \ref{oo}. In light of the conclusion of Lemma \ref{firstRT}, we will apply the Markov Additive Renewal Theorem, see for example Theorem 2.1 of Alsmeyer \cite{Alsmeyer1994}. 
We state the result below for completion in a form that is more appropriate for our purposes.

{\color{black}
\begin{theorem}[Alsmeyer \cite{Alsmeyer1994}]\label{alsthreorem}
Suppose we have the following  conditions:
\begin{itemize}
\item[(I)] The process $\Xi$  is an {\it aperiodic Harris recurrent} Markov chain, in the sense that there exists a probability measure, $\rho(\cdot)$ on $\mathcal{B}(\Omega)$ (Borel sets in $\Omega$) such that, for some $\lambda>0$,
\begin{equation}
{\mathbb P}^{\triangleleft}_{\theta}(\Xi_1 \in E) \geq \lambda \rho(E), \text{ for all  }\Omega, E\in\mathcal{B}(\Omega).
\label{Harris2}
\end{equation}
\item[(II)] Under $\mathbf{P}^\triangleleft$, $(\Xi_n, n\geq 0),$ has a stationary distribution, that is 
 \[
  \upsilon^*(\d\theta):= \lim_{n\to\infty}\mathbf{P}^\triangleleft_{0,\phi}(\Xi_n \in \d \theta), \qquad \theta\in\Omega, \phi\in \Omega, 
  \]
  exists as a non-degenerate distributional weak limit.
\item[(III)] With $\upsilon^*$ as above \begin{equation}
\mathbb{E}^\triangleleft_{\upsilon^*}[S_1]: = \int_{\Omega} \upsilon^*({\rm d} \theta)\mathbb{E}^{\triangleleft}_{\theta}[S_1]<\infty.
\label{Alsmeyer1}
\end{equation}
\item[(IV)] For any continuous  $f:\Gamma^3\to[0,\infty),$ such that $f$ is uniformly bounded, the mapping $r\mapsto G(r,\phi)$ is a.e. continuous, for any $\phi$ fixed, and
\begin{equation}
\int_{\Omega}  \int_0^{\infty} \upsilon^*(\d \phi)\sum_{n\geq 0} \sup_{nh<r\leq (n+1)h}G(r,\phi)<\infty,
\label{Alsmeyer2}
\end{equation}
for some $h>0.$
\end{itemize}
Then 
for $\upsilon^*$-a.e. $\Omega$, 
\begin{equation}
\label{Alsmeyer0}
\lim_{a\to0}\mathbb{E}^{\triangleleft}_{a\theta}\left[f(X_{\tau_1^{\ominus}},X_{\tau_1^{\ominus}-},X_{m(\tau_1^{\ominus}-)})\right]
= \frac{1}{\mathbb{E}^\triangleleft_{\upsilon^*}[S_1]}
\int_{\Omega}  \int_0^{\infty} \upsilon^*(\d \phi){\rm d} r \,G(r,\phi),
\end{equation}
where $G$ is as defined in (\ref{Gdef}). 
\end{theorem}
}

Before showing the conditions (I)-(IV) above hold, let us show that, a slightly enhanced version of Theorem \ref{alsthreorem} holds. {\color{black} More precisely we will show first the following corollary.
\begin{corollary}\label{enhancement}
Under the assumption that (\ref{Alsmeyer0}) and (I)-(II) are satisfied,
we can  remove the requirement that the limit is taken along the sequence of points $a\theta$, for $a\to0$ and $\upsilon^*$-a.e. $\theta$  in \eqref{Alsmeyer0}  and replace it by taking limits along $\Gamma\ni x\to0$. 
\end{corollary}

We note that this corollary, once proved,  also ends the proof of Theorem \ref{oo}.  }
In order to prove Corollary \ref{enhancement}, we must first recall the following lemma and deduce a relevant corollary which deals with the Boundary Harnack Principle. This will also be useful later on.

\smallskip

In the current setting, the Boundary Harnack Principle can be formulated as follows (see e.g. Bogdan et al. (BHP) in  \cite{BPW} and Bogdan \cite{Bogdan97}).

\begin{lemma}\label{L-BHP} Write $B_c: = \{x\in\mathbb{R}^d: |x|<c \}$ for the ball of radius $c>0$. Suppose that  $u,v:
\Gamma\to[0,\infty)$ are  functions satisfying
 $u(x)=v(x)=0$ whenever $x\in\Gamma^c \cap B_1$, and  are regular harmonic on $\Gamma\cap B_{1}$, meaning that, for each  $x\in\Gamma\cap B_1$,
\[
\mathbb{E}_x\left[u(X_{\tau_1^{\ominus}\wedge \kappa_\Gamma})\right] =u(x) \quad \text{ and } \quad  \mathbb{E}_x\left[v(X_{\tau_1^{\ominus}\wedge \kappa_\Gamma})\right] =v(x).
\]
Suppose, moreover, that  $u(x_0)=v(x_0)$ for some $x_0\in\Gamma\cap B_{1/2}$. Then, there exists a constant $C_1=C_1(\Gamma,\alpha)$ (which does not depend on the choice of $u$ or $v$) such  that,
\begin{equation}
C_1^{-1}v(x)\leq u(x)\leq C_1 v(x), \qquad x\in \Gamma\cap B_{1/2} .
\label{BHI}
\end{equation}
\end{lemma}

It is worth noting immediately that $M$ is a regular harmonic function on $\Gamma\cap B_{1}$ according to the above definition.
Indeed, from \eqref{MCOM}, the Optional Sampling Theorem and dominated convergence, using e.g. Theorem A in Blumenthal et al. \cite{BGR} which ensures $\mathbb{E}_x[|X_{\tau^\ominus_1}|^\beta]<\infty$, we know that 
\begin{align*}
M(x) &= \lim_{t\to\infty}\mathbb{E}_x\left[|X_{t\wedge \tau^\ominus_1}|^\beta M(\arg(X_{t\wedge \tau^\ominus_1})\mathbf{1}_{(t\wedge\tau^\ominus_1 < \kappa_\Gamma)}\right] \\
&= \mathbb{E}_x\left[M(X_{\tau^\ominus_1})\mathbf{1}_{(\tau^\ominus_1 < \kappa_\Gamma)}\right] \\
&= \mathbb{E}_x\left[M(X_{\tau^\ominus_1\wedge\kappa_\Gamma })\right]
\end{align*}
 As $M$ can only be defined up to a multiplicative constant, without loss of generality, we henceforth assume there is a $x_0\in\Gamma\cap B_{1/2}$, such that $M(x_0)=1$.

\begin{corollary}\label{BHP} Let $x_0$ be as above. For each $f\geq 0$ on $\mathbb{R}^d$ such that
\[
0<\mathbb{E}_{x_0}\left[f(X_{\tau_1^{\ominus}})\mathbf{1}_{(\tau_1^{\ominus}<\kappa_\Gamma)}\right]<\infty,
\]
there is a constant $C_1 = C_1(\Gamma,\alpha)$ (which does not depend on the choice of $f$) such that, for all $x\in \Gamma\cap B_{1/2}$,
$$  C_1^{-1}M(x)  \leq \frac{\mathbb{E}_x\left[f(X_{\tau_1^{\ominus}})\mathbf{1}_{(\tau_1^{\ominus}<\kappa_\Gamma)}\right]}{\mathbb{E}_{x_0}\left[f(X_{\tau_1^{\ominus}})\mathbf{1}_{(\tau_1^{\ominus}<\kappa_\Gamma)}\right] }  \leq C_1 M(x).$$
\end{corollary}
\begin{proof}
The result follows from Lemma \ref{L-BHP}, in particular from the inequalities \eqref{BHI}, as soon as we can verify that 
\[
g(x) := \frac{\mathbb{E}_x\left[f(X_{\tau_1^{\ominus}})\mathbf{1}_{(\tau_1^{\ominus}<\kappa_\Gamma)}\right]}{\mathbb{E}_{x_0}\left[f(X_{\tau_1^{\ominus}})\mathbf{1}_{(\tau_1^{\ominus}<\kappa_\Gamma)}\right]}  ,\qquad  x\in \mathbb{R}^d,
\] 
is regular harmonic on $\Gamma\cap B_1$. 

\smallskip

To this end, note that the function $g$ clearly vanishes on $\Gamma^C\cap B_1$ and is equal to a constant multiple of $f$ on $\Gamma\cap B_1^c$ by construction and (recalling that $M$ has been normalized so that $M(x_0) = 1$) we have $g(x_0)=M(x_0) = 1$ for some $x_0 \in\Gamma\cap B_{1/2}$. Finally, note $g(X_{\kappa_\Gamma}) = 0$ and $$g(X_{\tau^\ominus_1})\mathbf{1}_{(\tau^\ominus_1<\kappa_\Gamma)} = \frac{f(X_{\tau^\ominus_1})\mathbf{1}_{(\tau^\ominus_1<\kappa_\Gamma)}}{\mathbb{E}_{x_0}\left[f(X_{\tau_1^{\ominus}})\mathbf{1}_{(\tau_1^{\ominus}<\kappa_\Gamma)}\right]} $$ almost surely and hence, for $x\in \Gamma\cap B_1$,
\[
g(x) = \frac{\mathbb{E}_x\left[f(X_{\tau_1^{\ominus}})\mathbf{1}_{(\tau_1^{\ominus}<\kappa_\Gamma)}\right]}{\mathbb{E}_{x_0}\left[f(X_{\tau_1^{\ominus}})\mathbf{1}_{(\tau_1^{\ominus}<\kappa_\Gamma)}\right]}=\mathbb{E}_x\left[g(X_{\tau_1^{\ominus}})\mathbf{1}_{(\tau_1^{\ominus}<\kappa_\Gamma)}\right] =\mathbb{E}_x\left[g(X_{\tau_1^{\ominus}\wedge \kappa_\Gamma})\right], 
\]
as required.
\end{proof}

Let us now return to the promised proof of Corollary \ref{enhancement} which states  that we can take the limits  to the apex of $\Gamma$ in (\ref{Alsmeyer0}) in a more natural way.

%

\begin{proof}[Proof of Corollary \ref{enhancement}] For Borel sets $D\subseteq(0,\infty)$ and $E\in \Omega$, we have, with the help of scaling,
\begin{align}
\mathbb{P}^{\triangleleft}_{\theta}(S_1-1\in D, \Xi_1\in E)
& =\mathbb{P}^{\triangleleft}_{\theta/{\rm e}}(\log|X_{\tau_{1}^{\ominus}} |\in D, \arg(X_{\tau_{1}^{\ominus}})\in E)\notag\\
&=\frac{1}{M(\theta/{\rm e})}\mathbb{E}_{\theta/{\rm e}}\left[M(X_{\tau_1^{\ominus}})\mathbf{1}_{(\log|X_{\tau_{1}^{\ominus}} |\in D, \, \arg(X_{\tau_{1}^{\ominus}})\in E,\, \tau_1^{\ominus}<\kappa_{\Gamma})}\right].
\label{DE}
\end{align}
Note that, for $x\in\Gamma\cap B_1$, $D\in \mathcal{B}(0,\infty)$,  $E\in\mathcal{B}(\Omega),$
\[
g(x; D, E):=\mathbb{E}_{x}\left[M(X_{\tau_1^{\ominus}})\mathbf{1}_{(\log|X_{\tau_{1}^{\ominus}} |\in D, \, \arg(X_{\tau_{1}^{\ominus}})\in E,\, \tau_1^{\ominus}<\kappa_{\Gamma})}\right]
\]
is a regular harmonic function in $\Gamma\cap B_1$  (according to the definition in Lemma \ref{L-BHP}), which can be seen by applying the Strong Markov Property. Normalising $M$ so that $M(\theta_0/{\rm e}) = 1$, for some $\theta_0\in \Omega$, also tells us that 
\begin{equation}
\tilde\rho(D, E) = \mathbb{E}_{\theta_0/{\rm e}}\left[M(X_{\tau_1^{\ominus}})\mathbf{1}_{(\log|X_{\tau_{1}^{\ominus}} |\in D, \, \arg(X_{\tau_{1}^{\ominus}})\in E,\, \tau_1^{\ominus}<\kappa_{\Gamma})}\right] = g(\theta_0/{\rm e}; D, E)
\label{rho}
\end{equation}
is a probability distribution on $(0,\infty)\times\Omega$. For convenience, we shall shortly write 
$
\rho(E)
$
in place of $\tilde\rho((0,\infty), E)$ on $\Omega$.
\smallskip

Corollary \ref{BHP} now tells us that, for all $x\in \Gamma\cap B_{1/2}$,
\begin{equation}
C_1^{-1}M(x) \leq \frac{g(x; D, E)}{g(x_0; D, E)} \leq C_1M(x),
\label{rhoupper1} 
\end{equation}
where the constant $C_1\in(0,\infty)$ is universal and does not depend on the construction of $g$, nor $x_0$. Said another way, we have, for $\theta\in \Omega$ (which corresponds to $x = \theta/{\rm e}$ in \eqref{rhoupper1}), 
\begin{equation}
C_1^{-1}\tilde\rho(D, E) \leq \mathbb{P}^{\triangleleft}_{\theta}(S_1-1\in D, \Xi_1\in E)\leq \tilde\rho(D, E)C_1.
\label{rhoupper-new}
\end{equation}
\smallskip

Let us now assume that $A$ is a null set of $\upsilon^*$. From (II), we know that  ${\mathbb P}_{\upsilon^*}(\Xi_1 \in A) = \upsilon^*(A) =0$. On the one hand, from \eqref{rhoupper-new} and \eqref{DE}, we note that 
$0 = 
{\mathbb P}_{\upsilon^*}(\Xi_1 \in A) \geq C_1^{-1}\rho(A),
$ and hence that $\rho(A)=0.$ 
On the other hand, we know from \eqref{rhoupper-new} again together with the latter fact, that, ${\mathbb P}_{\theta}(\Xi_1 \in A) \leq C_1 \rho(A) =0$, for all $\theta\in\Omega$.
We have thus shown that the very first step of the process $\Xi$ positions it randomly so that it is in the concentration set of the support of $\upsilon^*$.

\smallskip

From \eqref{applyMAPRT} and \eqref{Vrenewal}, writing $I$ for the right-hand side of \eqref{Alsmeyer0}, we have, for $x\in \Gamma$ such that 
\begin{align}
&\left| \mathbb{E}^{\triangleleft}_x\left[f(X_{\tau_1^{\ominus}},X_{\tau_1^{\ominus}-},X_{m(\tau_1^{\ominus}-)})\right]-I \right|\notag\\
&\hspace{2cm}=\left|  {\mathbb E}^{\triangleleft}_{\arg(x)}\left[\sum_{n\geq 0} 
\mathbf{1}_{(S_n\leq -\log|x|)}G(-\log|x|-S_n,\Xi_n)\right] -I \right|\notag\\
&\hspace{2cm}\leq  G(-\log|x|,\arg(x))+\left| {\mathbb E}^{\triangleleft}_{\arg(x)}\left[\sum_{n\geq 1} \mathbf{1}_{(S_n\leq -\log|x|)} G(-\log|x|-S_n,\Xi_n)\right] -I\right|\notag\\
&\hspace{2cm}\leq G(-\log|x|,\arg(x))+{\mathbb P}^{\triangleleft}_{\arg(x)}(S_1> -\log|x|) I\notag\\
&\hspace{3cm}+{\mathbb E}^{\triangleleft}_{\arg(x)}\left[\mathbf{1}_{(S_1\leq -\log|x|)}\left| \mathbb{E}^{\triangleleft}_{\Xi_1{\rm e}^{S_1}}\left[\sum_{n\geq 0} \mathbf{1}_{(S_n\leq -\log|x|)}G(-\log|x|-{S}_n,{\Xi}_n)\right] -I\right|\right]
\label{twolims}
\end{align}
Taking care to note that 
\[
\mathbf{1}_{(\tau^\ominus_1\leq \tau_{{\rm e}|x|}^{\ominus})} = 
\left\{
\begin{array}{ll}
\mathbf{1}_{(\tau^\ominus_1= \tau_{{\rm e}|x|}^{\ominus})}=\mathbf{1}_{(|X_{\tau_{{\rm e}|x|}^{\ominus}}|>1)}&\text{ if }{\rm e}|x|<1\\
1&\text{ if }{\rm e}|x|\geq  1,
\end{array}
\right.
\]
we have for ${\rm e}|x|<1$ that
\begin{align*}
G(-\log|x|,\arg(x))&=\mathbb{E}^{\triangleleft}_{x}\left[f(X_{\tau_1^{\ominus}},X_{\tau_1^{\ominus}-},X_{m(\tau_1^{\ominus}-)})\mathbf{1}_{(\tau^\ominus_1\leq \tau_{{\rm e}|x|}^{\ominus})}
\right]\\
&\leq ||f||_\infty \mathbb{P}^{\triangleleft}_{x}(X_{\tau^\ominus_{{\rm e}|x|}}>1)
\\
&=||f||_\infty \mathbb{P}^{\triangleleft}_{\arg(x)}(X_{\tau^\ominus_{\rm e}}>{\color{black}1/|x|}) 
\\
&= ||f||_\infty  \mathbb{P}^{\triangleleft}_{\arg(x)}(S_1>-\log|x|).
\end{align*}
Hence, back in \eqref{twolims} we obtain 
\begin{align}
&\left| \mathbb{E}^{\triangleleft}_x\left[f(X_{\tau_1^{\ominus}},X_{\tau_1^{\ominus}-},X_{m(\tau_1^{\ominus}-)})\right]-I \right|\notag\\
&\hspace{1cm}\leq 
(||f||_\infty +I){\mathbb P}^{\triangleleft}_{\arg(x)}(S_1> -\log|x|)\notag\\
&\hspace{2cm}+{\mathbb E}^{\triangleleft}_{\arg(x)}\left[\mathbf{1}_{(S_1\leq -\log|x|)}\left| \mathbb{E}^{\triangleleft}_{\Xi_1{\rm e}^{S_1}}\left[\sum_{n\geq 0}\mathbf{1}_{(S_n\leq -\log|x|)} G(-\log|x|-{S}_n,{\Xi}_n)\right] -I\right|\right] 
\label{twolims2}
\end{align}
Appealing to  \eqref{rhoupper-new} we can develop the right-hand side of \eqref{twolims2} to get 
\begin{align}
&\left|\mathbb{E}^{\triangleleft}_x\left[f(X_{\tau_1^{\ominus}},X_{\tau_1^{\ominus}-},X_{m(\tau_1^{\ominus}-)})\right]-I\right|\notag\\
&\leq C_1(||f||_{\infty}+I)\rho((-\log|x|,\infty))\notag\\
&\hspace{1cm}
+C_1\int_0^{\infty}\int_{\Omega}\tilde\rho({\rm d}r,{\rm d} \phi)
\mathbf{1}_{(r\leq -\log|x|)}
\left|\mathbb{E}^{\triangleleft}_{|x|{\rm e}^r \phi}\left[f(X_{\tau_1^{\ominus}},X_{\tau_1^{\ominus}-},X_{m(\tau_1^{\ominus}-)})\right]-I\right|,\notag
\label{twolims}
\end{align}
Using that $\tilde\rho$ is a proper distribution, dominated convergence, the fact that $
{\mathbb P}_{\theta}(\Xi_1 \in \cdot) \ll \upsilon^*(\cdot)
$
and \eqref{Alsmeyer0} gives us that 
\[
\lim_{\Gamma\ni x\to0}\mathbb{E}^{\triangleleft}_x\left[f(X_{\tau_1^{\ominus}},X_{\tau_1^{\ominus}-},X_{m(\tau_1^{\ominus}-)})\right] =  \frac{1}{\mathbb{E}^\triangleleft_{\upsilon^*}[S_1]}
\int_{\Omega}  \int_0^{\infty} \upsilon^*(\d \phi){\rm d} r \,G(r,\phi),
\]
without restriction on $\arg(x)$ in relation to $\upsilon^*$, as $\Gamma\ni x\to0$.
\end{proof}

\subsection{Verification of conditions (I) and (II)}\label{sec:HR} 
We have already verified in \eqref{rhoupper-new} that 
\begin{equation}
{\mathbb P}^{\triangleleft}_{\theta}(\Xi_1 \in E)\geq C^{-1}_1 \rho(E),
\label{rhoupper2}
\end{equation}
which implies that  \eqref{Harris2} holds, i.e. condition (I) is holds.

\smallskip

The following lemma follows directly from  \eqref{rhoupper2} and is a rewording of e.g. Theorems VII.3.2 and VII.3.6 of Asmussen \cite{AsmussenQueue}, and it addresses precisely condition (II).

\begin{lemma}\label{pi*}
Under $\mathbf{P}^\triangleleft$, $(\Xi_n, n\geq 0),$ has a stationary distribution, that is 
 \[
  \upsilon^*(\d\theta):= \lim_{n\to\infty}\mathbf{P}^\triangleleft_{0,\phi}(\Xi_n \in \d \theta), \qquad \theta\in\Omega, \phi\in \Omega, 
  \]
  exists as a non-degenerate distributional weak limit. Hence, the condition (II) is satisfied. 
\end{lemma}
\begin{remark}\label{upsabscts}
Note that  
\[
\int_\Omega \upsilon^*(\d\theta)\mathbb{P}_{\theta}(\Xi_n \in \d\phi, \, T_n<\kappa_\Gamma) \frac{M(\phi)}{M(\theta)}= \upsilon^*(\d\phi),\qquad 
\]
which makes  $\upsilon^\Gamma(\d \phi) = \upsilon^*(\d \phi)/M(\phi)$, $\phi\in\Omega$, an invariant measure for the killed semigroup 
$\mathbb{P}_{\theta}(\Xi_n \in \d\phi, \, T_n<\kappa_\Gamma)$, $n\ge 0$.
\end{remark}

Note that, under the assumptions (I) and (II),  the limiting distribution \eqref{Alsmeyer0} is proper, which can be seen by taking $f =1$, in which case
\begin{equation}
\begin{split}
\label{finitepp}
\int_{\Omega}  \int_0^{\infty} \upsilon^*(\d \phi){\rm d} r \,G(r,\phi) 
&= \int_{\Omega}  \int_0^{\infty} \upsilon^*(\d \phi){\rm d} r
\mathbb{P}^{\triangleleft}_{{\rm e}^{-r} \phi}(
\tau^\ominus_1\leq  \tau_{{\rm e}^{1-r}}^{\ominus}
)\\
&= \int_{\Omega}  \int_0^{1} \upsilon^*(\d \phi){\rm d} r
+ \int_{\Omega}  \int_1^\infty \upsilon^*(\d \phi){\rm d} r
\mathbb{P}^{\triangleleft}_{{\rm e}^{-r} \phi}(\tau^\ominus_1= \tau_{{\rm e}^{1-r}}^{\ominus}
)\\
&= 1+\int_{\Omega}  \int_1^{\infty} \upsilon^*(\d \phi){\rm d} r\,
\mathbb{P}^{\triangleleft}_{{\rm e}^{-r}\phi}(
\log|X_{\tau_{{\rm e}^{1-r}}^{\ominus}}| > 0
)\\
&=1+ \int_{\Omega}  \int_1^{\infty} \upsilon^*(\d \phi){\rm d} r\,
\mathbb{P}^{\triangleleft}_{\phi}(
\log|{\rm e}^{-r}X_{\tau_{{\rm e}}^{\ominus}}| > 0
)\\
&=1+ \int_{\Omega}  \int_1^{\infty} \upsilon^*(\d \phi){\rm d} r\,
\mathbb{P}^{\triangleleft}_{\phi}(S_1 > r
)\\
&=1+\mathbb{E}^\triangleleft_{\upsilon^*}[S_1-1]\\
&=\mathbb{E}^\triangleleft_{\upsilon^*}[S_1]
\end{split}
\end{equation}
 and hence the limit on the right-hand side of \eqref{Alsmeyer0}
is equal to unity (for $f=1$).

\subsection{Verification of conditions (III) and (IV)} We do this with two individual lemmas. 

\begin{lemma}\label{verifyAlsmeyer1}
Condition (III) holds, i.e. $\mathbb{E}^\triangleleft_{\upsilon^*}[S_1]<\infty$.
 \end{lemma}
\begin{proof} We can appeal to the law of first exit from a sphere given in Theorem A of Blumenthal et al. \cite{BGR} to deduce that, up to constant $C$, which is irrelevant for our computations, and may take different values in each line of the below computation, we have the following inequalities
\begin{align}
&\sup_{|x|<1/2}\mathbb{E}_{x}\left[M(X_{\tau_1^{\ominus}})(1+\log|X_{\tau_1^{\ominus}}|)
\mathbf{1}_{(\tau_1^{\ominus}<\kappa_{\Gamma})}\right]\notag\\
&\leq \sup_{|x|<1/2}\mathbb{E}_{x}[M(X_{\tau_1^{\ominus}})(1+\log|X_{\tau_1^{\ominus}}|)]\notag\\
&=C\sup_{|x|<1/2}\int_{|y|>1}\d y(|1-|x|^2)^{\alpha/2}(|y|^2-1)^{-\alpha/2}M(y) \frac{1+\log|y|}{|y-x|^d}\notag\\
&\leq C\sup_{|x|<1/2}\int_{\Omega}{\rm d}\theta M(\theta) \int_{1}^\infty {\rm d} r(r^2-1)^{-\alpha/2}r^{d-1+\beta}\frac{1+\log r}{|r\theta-x|^d}\notag\\
&=C\sup_{|x|<1/2}\int_{\Omega}{\rm d}\theta M(\theta) \int_{1}^2 {\rm d} r(r^2-1)^{-\alpha/2}r^{d-1+\beta}\frac{1+\log r}{|r\theta-x|^d}\notag\\
&\qquad +C\sup_{|x|<1/2}\int_{\Omega}{\rm d}\theta M(\theta) \int_{2}^\infty {\rm d} r(r^2-1)^{-\alpha/2}r^{d-1+\beta}\frac{1+\log r}{|r\theta-x|^d}\notag\\
&=:B_1+B_2.
\label{finiteness}
\end{align}
Using that $|r\theta-x|\geq r-|x|\geq 1/2$ we can bound the first term as follows:
\begin{equation*}
\begin{split}
B_1&\leq 2^{2d-1+\beta}(1+\log 2)C\int_{\Omega}{\rm d}\theta M(\theta)\int_{1}^2 {\rm d} r(r^2-1)^{-\alpha/2}<\infty.
\end{split}
\end{equation*}
To verify that the second term in \eqref{finiteness} is finite also, we use that $|r\theta-x|\geq {3r}/{4},$ by the triangle inequality, and that necessarily $\beta<\alpha,$ to obtain that 
\begin{equation*}
\begin{split}
B_2 &\leq \left(\frac{4}{3}\right)^{d+\frac{\alpha}{2}}C\int_{\Omega}{\rm d}\theta M(\theta) \int_{2}^\infty{\rm d} r \,r^{\beta-\alpha-1}(1+\log r)<\infty,
\end{split}
\end{equation*}
We can now apply the Boundary Harnack Principle in Corollary \ref{BHP} and the scaling property to deduce that 
\begin{align*}\int_{\Omega} \upsilon^*({\rm d} \theta)\mathbb{E}^{\triangleleft}_{\theta}[S_1]&=\int_{ \Omega} \upsilon^*({\rm d} \theta)\mathbb{E}^{\triangleleft}_{\theta}[\log|X_{\tau_{\rm e}^{\ominus}}|]\\
&=\int_{\Omega} \upsilon^*({\rm d} \theta)\mathbb{E}^{\triangleleft}_{\theta/{\rm e}}[\log|X_{\tau_1^{\ominus}}|+1]\\
&=  \int_{\Omega} \upsilon^*({\rm d} \theta)\frac{\mathbb{E}_{\theta/{\rm e}}\left[\mathbf{1}_{(\tau^\ominus_1<\kappa_\Gamma)}M(X_{\tau_1^{\ominus}})(\log|X_{\tau_1^{\ominus}}|+1)\right]}{M(\theta/{\rm e})}\\
&< \int_{\Omega} \upsilon^*({\rm d} \theta) C_1\sup_{|x|<1/2}\mathbb{E}_{x}\left[\mathbf{1}_{(\tau^\ominus_1<\kappa_\Gamma)}M(X_{\tau_1^{\ominus}})(1+\log|X_{\tau_1^{\ominus}}|)\right]\\
&<\infty,
\end{align*}
where finiteness follows from \eqref{finiteness}.
\end{proof}
\begin{lemma}\label{lemma:IVholds} The conditions in (IV) holds.
\end{lemma}
\begin{proof}
Let $f:\Gamma^3\to [0,\infty)$ be a continuous and bounded function. On account of continuity of  $M$  and standard Skorokhod continuity properties of the stable process with killing at first passage times, together with the dominated convergence theorem, imply that for any $\phi\in \Gamma$ fixed, the function $$y\mapsto G(y,\phi):=\mathbb{E}^{\triangleleft}_{{\rm e}^{-y} \phi}\left[f(X_{\tau_1^{\ominus}},X_{\tau_1^{\ominus}-},X_{m(\tau_1^{\ominus}-)})\mathbf{1}_{(\tau^\ominus_1\leq \tau_{{\rm e}^{1-y}}^{\ominus})}\right], \qquad y>0,$$ is continuous and bounded. Since $f$ is assumed to be bounded it is enough to check that \eqref{Alsmeyer2} holds with $f\equiv 1.$ But this follows from a straightforward modification of the computation in \eqref{finitepp}, using that for any $\theta\in \Gamma$ fixed, the function $r\mapsto \mathbb{P}^{\triangleleft}_{\phi}(S_1 > r)$ is non-increasing, together with the conclusion of Lemma~\ref{verifyAlsmeyer1}.  
\end{proof}

\section{Proof of Theorem \ref{ooC}}\label{proveooC}
Let us define a new family of stopping times with respect to the filtration generated by $((H^+_t, \theta^+_t),\ t\geq 0)$. Set $\chi_0 = 0$ and 
\[
\chi_{n+1} =\inf\{s>\chi_n : H^+_s -H^+_{\chi_n}>1\}, \qquad n\geq0.
\]
We should also note that these stopping times have the property that the sequence of pairs $((S_n, \Xi_n),\ n\geq 0)$, agrees precisely with $((H_{\chi_n}, \Theta^+_{\chi_n}),\ n\geq 0)$. Moreover, it is easy to show that $((\chi_n, \Xi_n),\ n\geq 0)$, is a Markov additive process, and we known $(\Xi_n, n\geq0)$ is Harris recurrent, in the sense of (I) above. 

Let, 
\[
\mathscr{U}^\triangleleft_\theta({\rm d}s,{\rm d}\phi):=  \sum_{n\geq 0} \mathbf{P}_{\theta}^{\triangleleft} (\chi_n\in {\rm d} s,\Xi_n \in {\rm d}\phi), \qquad s\geq 0, \Omega.
\]
Appealing to the Markov property, we have, for $\Omega$ and bounded measurable $f$ on $\Omega$, 
\begin{align*}
\mathbf{E}_{0,\theta}^{\triangleleft}[f(\Theta^+_t)]
&=\mathbf{E}_{0,\theta}^{\triangleleft} \left[\sum_{n\geq 0}\mathbf{1}_{(\chi_n\leq  t< \chi_{n+1})}f(\Theta^+_t)\right] \\
&=\mathbf{E}_{0,\theta}^{\triangleleft}\left[\sum_{n\geq 0}\mathbf{1}_{(\chi_n\leq  t)}\mathbf{E}^\triangleleft_{0,\phi}\left[\mathbf{1}_{( u <{\chi}_{1})}f(\Theta^+_{u})\right]_{\phi = \Theta^+_{\chi_n}, u = t- \chi_n}\right] \\
&=\int_{0}^t\int_{ \Omega} \mathscr{U}^\triangleleft_\theta({\rm d}s,{\rm d}\phi)F(t-s,\phi),
\end{align*}
with $F(s,\phi)=\mathbf{E}^{\triangleleft}_{0,\phi}[\mathbf{1}_{( s \leq \chi_{1})}f(\Theta^+_{s})]$ which is bounded and continuous in both its arguments.
Note, moreover, that
\begin{align*}
\int_{0}^\infty\int_{ \Omega} \upsilon^*(\d \phi)\d s\, F(s, \phi) &= \int_{0}^\infty\int_{ \Omega} \upsilon^*(\d \phi)\d s\, \mathbf{E}^{\triangleleft}_{0,\phi}[\mathbf{1}_{( s \leq \chi_{1})}f(\Theta^+_{s})]\\
&=\int_{0}^\infty\int_{ \Omega} \upsilon^*(\d \phi)\d s\, \mathbf{E}^{\triangleleft}_{0,\phi}[\mathbf{1}_{(H^+_s<1)}f(\Theta^+_{s})]\d s\\
&=\int_\Omega\int_\Omega \upsilon^*(\d \phi)U^\triangleleft_{\phi}([0,1), \d\theta)f(\theta),
\end{align*}
where $U^\triangleleft_{\phi}(\d x,\d \theta)$, $x\geq 0$, $\Omega$, is the ascending ladder MAP potential 
\[
U^\triangleleft_{\phi}(\d x,\d\theta) = \int_0^\infty \mathbf{P}^\triangleleft_{0,\phi}(H^+_{s}\in \d x, \Theta^+_s\in \d \theta)\d s.
\]
As such, whenever $f$ is bounded, we have that $\int_{0}^\infty\int_{ \Omega} \upsilon^*(\d \phi)\d s\, F(s, \phi)<\infty$.
\smallskip

We also note that 
\begin{align*}
\mathbf{E}_{0,\upsilon^*}^{\triangleleft}[\chi_1]&: = \int_\Omega\upsilon^*(\d \phi)\mathbf{E}_{0,\phi}^{\triangleleft}[\chi_1]\\
&=\int_\Omega\upsilon^*(\d \phi)\int_0^\infty \mathbf{P}_{0,\phi}^{\triangleleft}(\chi_1>t)\d t\\
&=\int_\Omega\upsilon^*(\d \phi)\int_0^\infty \mathbf{P}_{0,\phi}^{\triangleleft}(H^+_t <1)\d t\\
&=\int_\Omega \upsilon^*(\d \phi)U^\triangleleft_{\phi}([0,1), \Omega)<\infty.
\end{align*}
Arguing as in the proof of Lemma~\ref{lemma:IVholds}, it follows that whenever $f$ is continuous and bounded, the mapping $(s,\phi)\mapsto F(s, \phi),$ satisfies the conditions in Theorem 2.1 in \cite{Alsmeyer1994}. 
\smallskip

As such, and on account of the fact that $(\Xi_n, n\geq 0)$ has been proved to have a stationary distribution, $\upsilon^*$, we can again invoke the Markov additive renewal theorem \cite{Alsmeyer1994} and conclude that, for $\upsilon^*$-almost every $\Omega$,
$$\lim_{t\to\infty}\mathbf{E}_{0,\theta}^{\triangleleft}[f(\Theta^+_t)]=\frac{1}{\mathbf{E}_{0,\upsilon^*}^{\triangleleft}[\chi_1]}\int_\Omega\int_\Omega \upsilon^*(\d \phi)U^\triangleleft_{\phi}([0,1), \d\theta)f(\theta).$$
We can upgrade the previous statement to allow for all $\Omega$ by appealing to reasoning  similar in fashion to the proof of Corollary \ref{enhancement}. For the sake of brevity, we thus leave this as an exercise for the reader.

\smallskip

In conclusion, $(\Theta^+_t, t\geq 0)$ has a  non-degenerate stationary distribution, which is given by 
\[
\pi^{\triangleleft,+}(\d \theta) =\frac{\int_\Omega\int_\Omega \upsilon^*(\d \phi)U^\triangleleft_{\phi}([0,1), \d\theta)}{\int_\Omega \upsilon^*(\d \phi)U^\triangleleft_{\phi}([0,1), \Omega)}, \qquad \theta\in\Omega.
\]
as required.\hfill$\square$

\section{Proof of Theorem \ref{conditionapex}} \label{proofconditionapex}We first need a technical Lemma. Recall that $\tau^\oplus_a : = \inf\{t>0: |X_t|<a\}$, $a>0$.
\begin{lemma}\label{decay} We have the following convergence,
$$
\lim_{\Gamma\ni aKx\to0}\frac{\mathbb{P}^\triangleleft_x(\tau^\oplus_a<\infty)}{(|x|/ a)^{\alpha - 2\beta - d} }=\frac{1}{\mathbb{E}^\triangleleft_{\upsilon^*}[\log|X_{\tau^\ominus_{\rm e}}|]}
\int_{\Omega}  \int_0^{\infty} \upsilon^*(\d \phi){\rm d} r \,\mathbb{E}_{{\rm e}^{-r} \phi}^\triangleleft\left[ |X_{\tau_1^{\ominus}}|^{\alpha-2\beta-d}\right]<\infty.$$

\end{lemma}
\begin{proof} {\color{black}We first use properties from the Riesz--Bogdan--\.Zak transform in Theorem \ref{RBSthrm} and the scaling properties \eqref{1/a} and  \eqref{stopscale} (employed similarly for $\tau^{\ominus}_{1/a}$) to deduce that 
\begin{align*}
\mathbb{P}^\triangleleft_x(\tau^\oplus_a <\infty)&=\mathbb{E}_x\left[ \frac{M(X_{\tau^\oplus_a})}{M(x)};\tau^\oplus_a<\kappa_\Gamma\right]\\
& =\mathbb{E}^\circ_{Kx} \left[\frac{M(K X_{\tau^\ominus_{1/a}})}{ M(\arg(x))|x|^\beta} ;\tau^\ominus_{1/a}<\kappa_\Gamma\right]\\
&=\mathbb{E}_{Kx}\left[ \frac{ |X_{\tau^\ominus_{1/a}}|^{\alpha-\beta -d }M(\arg(X_{\tau^\ominus_{1/a}} ))}{ |x|^{-(\alpha-\beta -d)}M(\arg(x))};\tau^\ominus_{1/a} <\kappa_\Gamma\right]\\
&=\mathbb{E}_{aKx}\left[ \frac{ |X_{\tau^\ominus_{1}} /a|^{\alpha-\beta -d }M(\arg(X_{\tau^\ominus_{1}} ))}{ |x|^{-(\alpha-\beta -d)}M(\arg(x))};\tau^\ominus_{1} <\kappa_\Gamma\right]\\
&=\frac{|x|^{\alpha - 2\beta-d}}{a^{\alpha-2\beta-d}}\mathbb{E}_{aKx}^\triangleleft\left[ |X_{\tau^\ominus_{1}} |^{\alpha-2\beta -d }\right]. 
\end{align*}

Using Theorem \ref{oo} with $f(x) = |x|^{\alpha-2\beta -d}\mathbf{1}_{(|x|\geq 1 )}$, we thus have that 
\begin{align}
\lim_{\Gamma\ni aKx\to0}&
\frac{\mathbb{P}^\triangleleft_x(\tau^\oplus_a <\infty)}{(|x|/a)^{\alpha-2\beta-d}}\notag\\
&=\lim_{\Gamma\ni aKx\to0}\mathbb{E}^\triangleleft_{aKx}\left[ |X_{\tau^\ominus_{1}}|^{\alpha-2\beta-d}\right]\notag\\
&=\frac{1}{\mathbb{E}^\triangleleft_{\upsilon^*}[\log|X_{\tau^\ominus_{\rm e}}|]}
\int_{\Omega}  \int_0^{\infty} \upsilon^*(\d \phi){\rm d} r \,\mathbb{E}_{{\rm e}^{-r} \phi}^\triangleleft\left[ |X_{\tau_1^{\ominus}}|^{\alpha-2\beta - d} \mathbf{1}_{(\tau^\ominus_1\leq \tau_{{\rm e}^{1-r}}^{\ominus})}\right]<\infty
  \label{bounded}
\end{align}
where we have used that $\alpha-2\beta -d<0$,  $|X_{\tau^\ominus_1}|\geq 1$ and Theorem~\ref{oo}  with $f(x) = |x|^{\alpha-2\beta - d}\mathbf{1}_{(|x|\geq 1 )}$, for $x\in\Gamma$. 
The result now follows.
}
\end{proof}

{\color{black}
Returning now to the proof of Theorem \ref{conditionapex},  the usual application of the strong Markov property means we need to  evaluate, for $x\in\Gamma$,
\begin{equation} \mathbb{P}^{\triangleright}_x(A, \, t<\kappa_{\{0\}}) =\lim_{a\to0} \mathbb{E}^{\triangleleft}_x\left[\mathbf{1}_{(A,\, t< \tau^\oplus_a)}\frac{\mathbb{P}^\triangleleft_{X_t}(\tau^\oplus_a <\infty)}{\mathbb{P}^\triangleleft_x(\tau^\oplus_a<\infty)}\right], 
\label{useDCT}
\end{equation}
where $A\in \mathcal{F}_t$. 
In order to do so, we first note from Lemma \ref{decay} that, 
\[
\lim_{a\to0}\frac{\mathbb{P}^\triangleleft_{X_t}(\tau^\oplus_a <\kappa_\Gamma)}{\mathbb{P}^\triangleleft_x(\tau^\oplus_a<\kappa_\Gamma)} = \frac{|X_t|^{\alpha -2\beta-d}}{|x|^{\alpha - 2\beta-d}}.
\]

Moreover, from \eqref{bounded},  we also see that, for each $\varepsilon>0$, there exists a constant $\Delta>0$ such that, when 
$
|aKx| = (a/|x|)<\Delta,
$ 
\[
(1-\varepsilon)C_1\leq  \frac{\mathbb{P}^\triangleleft_x(\tau^\oplus_a <\infty)}{(|x|/a)^{\alpha-2\beta-d}}\leq (1+\varepsilon)C_1.
\] With $C_1$ as in the previous Lemma. 
On the other hand, if $(a/|x|)\geq \Delta$, then 
\[
\mathbb{P}^\triangleleft_x(\tau^\oplus_a <\infty) \leq 1 \leq (a/|x|)^{2\beta+d-\alpha}\Delta^{\alpha-2\beta-d}.
\]
From this we conclude that there is an appropriate choice of constant $C$ such that for $a\ll1$,
\begin{equation}\label{bound}
\frac{\mathbb{P}^\triangleleft_{X_t}(\tau^\oplus_a <\infty)}{\mathbb{P}^\triangleleft_x(\tau^\oplus_a <\infty)} \leq C|X_t|^{\alpha-2\beta-d}.
\end{equation}
We want to show that 
\begin{align}
\mathbb{E}^{\triangleleft}_x[|X_t|^{\alpha -2\beta-d}]
&=\mathbb{E}_x\left[\frac{M(\arg(X_t))}{M(x)}|X_t|^{\alpha -\beta-d}\mathbf{1}_{(t<\kappa_\Gamma)}\right] \notag\\
&\leq \frac{C}{M(x)}\mathbb{E}_x\left[|X_t|^{\alpha -\beta-d}\mathbf{1}_{(t<\kappa_\Gamma)}\right]<\infty.
\label{DCTuse}
\end{align} 

To this end, we note that, since $\alpha -\beta-d<0$,  $\mathbb{E}_x[|X_t|^{\alpha -\beta-d}\mathbf{1}_{(|X_t|
\geq 1, \, t<\kappa_\Gamma)}]\leq 1$.
The problem thus lies in showing that 
$
\mathbb{E}_x[|X_t|^{\alpha  -\beta-d}\mathbf{1}_{(|X_t|< 1, \, t<\kappa_\Gamma)}]<\infty.
$
To this end, let us recall from Bogdan et al. \cite{BPW} that $(X^\Gamma_t, t\geq 0)$, the stable process killed on exiting $\Gamma$, has a semigroup density, say $p^\Gamma_t(x,y)$, $x,y\in\Gamma$. Moreover, in equation (10) and (53) of the aforesaid reference, they showed that 
\[
p^\Gamma_t(x,t)
\approx 
\mathbb{P}^\Gamma_x(\kappa_\Gamma>t) \mathbb{P}^\Gamma_y(\kappa_\Gamma>t) \left(t^{1/\alpha}\wedge\frac{t}{|y|^{\alpha+d}} \right), \qquad x,y\in\Gamma,
\]
where $p^{\Gamma}_t(x,y)$ is the transition density of $(X, \mathbb{P}^{\Gamma})$ and $f(x,t)\approx g(x,t)$ means that, uniformly in the domains of $f$ and $g$, there exists a constant $c>0$ such that $c^{-1}\leq f/g\leq c$.
It thus follows that 
\begin{align*}
\mathbb{E}_x&[|X_t|^{\alpha  -\beta-d}\mathbf{1}_{(|X_t|< 1, \, t<\kappa_\Gamma)}]\\
&\leq C\mathbb{P}^\Gamma_x(\kappa_\Gamma>t) \int_{|y|\leq 1}\mathbb{P}^\Gamma_y(\kappa_\Gamma>t) \left(t^{1/\alpha}\wedge\frac{t}{|y|^{\alpha+d}} \right)|y|^{\alpha-\beta-d}\d y\\
&\leq C t^{1/\alpha}\mathbb{P}^\Gamma_x(\kappa_\Gamma>t) \int_{0}^1
r^{\alpha-\beta-1}\d r <\infty,
\end{align*}
 using $\alpha >\beta$, where the  constant $C$ has a different value in each line of the calculation above, but otherwise is unimportant.
\smallskip

The bound \eqref{bound} and the finite moment \eqref{DCTuse} can now be used in conjunction with  the Dominated Convergence Theorem (for the inner limit) followed by the Monotone Convergence Theorem (for the outer limit) in \eqref{useDCT} to deduce 
\begin{align}
\mathbb{P}^{\triangleright}_x(A, \, t<\kappa_{\{0\}})
&=\mathbb{E}^\triangleleft_x\left[
\mathbf{1}_{A}\frac{|X_t|^{\alpha -2\beta-d}}{|x|^{\alpha - 2\beta-d}}
\right]=\mathbb{E}_x\left[
\mathbf{1}_{(A, \, t<\kappa_\Gamma)}\frac{H(X_t)}{H(x)}
\right], \qquad t\geq 0,
\label{COMH2}
\end{align}
with $H(x) = |x|^{\alpha - \beta -d}M(\arg(x))$ as required.
\smallskip

We must also show that this process is continuously absorbed at $0$.  Noting that $H(x) = |x|^{\alpha - d}M(Kx)$, applying the Riesz--Bogdan--\.Zak transform (cf. Theorem \ref{RBSthrm}), for continuous and bounded $f:\mathbb{R}^d\times\mathbb{R}^d\to[0,\infty)$ and $0<a<|x|$,

\begin{align*}
\mathbb{E}^\triangleright_x[f(X_{\underline{m}(\tau^\oplus_a-)}, X_{\tau^\oplus_a})]& 
=\mathbb{E}^\circ_x\left[f(X_{\underline{m}(\tau^\oplus_a-)}, X_{\tau^\oplus_a})
 \mathbf{1}_{(\tau^\oplus_a<\kappa_\Gamma)}\frac{M(KX_{\tau^\oplus_a})}{M(Kx)}\right]\\
 & 
=\mathbb{E}_{Kx}\left[f(KX_{\overline{m}(\tau^\ominus_{1/a}-)}, KX_{\tau^\ominus_{1/a}})
 \mathbf{1}_{(\tau^\ominus_{1/a}<\kappa_\Gamma)}\frac{M(X_{\tau^\ominus_{1/a}})}{M(Kx)}\right]\\
 &=\mathbb{E}^\triangleleft_{Kx}\left[f(KX_{\overline{m}(\tau^\ominus_{1/a}-)}, KX_{\tau^\ominus_{1/a}})
\right],
\end{align*}
where, for $a>0$, $\underline{m}(\tau^\oplus_{a}-) = \sup\{t<\tau^\oplus_{a} : |X_t| = \inf_{s<t}|X_s|\}$ and 
$\overline{m}(\tau^\ominus_{a}-) = \sup\{t<\tau^\ominus_{a} : |X_t| = \inf_{s<t}|X_s|\}$.
From Theorem \ref{oo}  it follows that the limit on the right-hand side above is equal to $f(0,0)$. This shows $(X,\mathbb{P}^{\triangleright}_x)$, $x\in\Gamma$  is almost surely absorbed continuously at 0.

\smallskip

Finally, reconsidering the proof of Proposition \ref{jumprate}, the remaining statement is straightforward to prove in the same way. This concludes the proof of Theorem \ref{conditionapex}.
}
\hfill$\square$

\section{Proof of Theorem \ref{BZcone}}\label{proofBZcone} It turns out more convenient to prove Theorems \ref{BZcone} 
before we deal with Theorem \ref{0}.  Indeed, it will play a crucial role in its proof.

\subsection{Proof of Theorem \ref{BZcone} (i)} We can verify the statement of this part of the theorem by first noting that the transformation $(KX_{\eta(t)}, t\geq 0)$ maps $(X, \mathbb{P}^\triangleleft)$, to a new self-similar process. Then we verify it has the transitions of $(X, \mathbb{P}^\triangleright)$.

\smallskip

For the first of the aforesaid, we refer back to the Lamperti--Kiu transform. As already observed in Alili et al. \cite{ACGZ} and Kyprianou \cite{ALEAKyp}, from the Lamperti--Kiu representation of $(X, \mathbb{P}^\triangleleft)$, 
\[
KX_{\eta(t)} = {\rm e}^{-\xi_{\varphi\circ\eta(t)}}\Theta_{ \varphi\circ\eta(t) }
\qquad t\geq0.
\]
Note however that 
\[
\int_0^{\varphi(t)}{\rm e}^{\alpha{\xi}_{s}}{\rm d}s = t \text{ and  }
\int_0^{\eta(t)} {\rm e}^{-2\alpha{\xi}_{\varphi(u)}}
{\rm d}u = t,
\qquad t\geq0.
\]
A straightforward differentiation of the last two integrals shows  that, respectively, 
\[
\frac{{\rm d}\varphi(t)}{{\rm d}t} = {\rm e}^{-\alpha {\xi}_{\varphi(t)}}\text{ and }
\frac{{\rm d}\eta(t)}{{\rm d}t} ={\rm e}^{2\alpha {\xi}_{\varphi\circ\eta(t)}}, \qquad t\geq0,
\]
and so 
the chain rule now tells us 
\begin{equation}
\frac{{\rm d}(\varphi\circ\eta)(t)}{{\rm d}t} = \left.\frac{{\rm d}\varphi(s)}{{\rm d}s}\right|_{s = \eta(t)}\frac{{\rm d}\eta(t)}{{\rm d}t}  = {\rm e}^{\alpha {\xi}_{\varphi\circ\eta(t)}},
\label{repeatinddim}
\end{equation}
and hence,
$
\varphi\circ\eta(t) = \inf\left\{s>0: \int_0^s {\rm e}^{-\alpha \xi_u}\d u >t\right\}.
$
It is thus clear that $(KX_{\eta (t)}, t\geq 0)$ is a self-similar Markov process with underlying MAP equal to $(-\xi, \Theta)$. 

\smallskip

To verify it has the same transitions as $(X, \mathbb{P}^\triangleright)$, we note that $(\eta(t),\ t\geq 0)$, is a sequence of stopping times
{\color{black}\begin{align*}
\mathbb{E}^{\triangleleft}_{Kx}&[f(KX_{\eta(t)})] \\
&=  \mathbb{E}_{Kx}\left[\frac{M(X_{\eta(t)})}{M(Kx)}f(KX_{\eta(t)});\,\eta(t)<\kappa_\Gamma\right] \\
& =\mathbb{E}_{Kx}\left[\frac{|KX_{\eta(t)}|^{-\beta}M(\arg(KX_{\eta(t)}))}{|x|^{-\beta}M(\arg(x))}f(KX_{\eta(t)});\, \eta(t)<\kappa_\Gamma\right] \\
& =\mathbb{E}_{x}\left[\frac{|X_t|^{\alpha-d}}{|x|^{\alpha-d}}\frac{|X_{t}|^{-\beta}M(\arg(X_{t}))}{|x|^{-\beta}M(\arg(x))}f(X_{t}); t<\kappa_\Gamma\right] \\
& =\mathbb{E}_{x}\left[\frac{|X_{t}|^{\alpha-d-\beta}M(\arg(X_{t}))}{|x|^{\alpha-d-\beta}M(\arg(x))}f(X_{t}); t<\kappa_\Gamma\right]\\
&=\mathbb{E}^\triangleright_{x}\left[f(X_{t})\right]
\end{align*}
where in the third equality we have applied the regular Riesz--Bogdan--\.Zak transform (cf. Theorem \ref{RBSthrm}) and in the final equality we have used 
Theorem \ref{conditionapex}.\hfill$\square$
}

\subsection{Proof of Theorem \ref{BZcone} (ii)}  The proof of this part appeals to Theorem 3.5 of Nagasawa \cite{Nagasawa1964}. The aforesaid classical result gives directly the conclusion of part (ii) as soon as a number of conditions are satisfied. Most of the conditions are trivially satisfied thanks to the fact that $(X, \mathbb{P}^\triangleleft)$,  is a regular Markov process (see for example the use of this Theorem in Bertoin and Savov \cite{BS} or D\"oring and Kyprianou \cite{DK}). However the two most important conditions stand out as non-trivial and require verification here. 

\smallskip

In the current context, the first condition requires the existence of a sigma-finite measure $\mu$ such that the duality relation is satisfied,  for any $f,g:\Gamma\to\mathbb{R}$ measurable and bounded, one has
\begin{equation}
\int_{\Gamma}\mu(\d x)f(x)\int_{\Gamma}\d y\,  p^\triangleleft_t(x,y)g(y) =\int_{\Gamma}\mu(\d x)g(x)\int_{\Gamma}\d y \, p^\triangleright_t(x,y)f(y),
\qquad \forall t\geq 0,
\label{switching}
\end{equation}
and the second requires that 
\begin{equation}
\mu(\d x) = G^\triangleleft(0,\d x): = \int_0^\infty \mathbb{P}_0^\triangleleft(X_t \in\d x)\d t, \qquad x\in\Gamma.
\label{referencemeasure}
\end{equation}
Our immediate job is thus to understand the analytical shape of the measure $\mu$.
To this end,  we prove the following intermediary result, the conclusion of which automatically deals with \eqref{referencemeasure}.
\begin{lemma}\label{excursionpotential}
We have for bounded and  measurable  $f:\Gamma\to[0,\infty)$, which is compactly supported in $\Gamma$,  up to a multiplicative constant, 
\[
\int_\Gamma f(x)G^\triangleleft(0,\d x) = \int_\Gamma f(x)M(x)H(x)\d x.
\]
\end{lemma}
\begin{proof} Referring to some of the facts displayed in Theorem \ref{entrance}, we have with the help of Fubini's Theorem, the scaling properties of the transition density $p^\Gamma$,  and \eqref{P0} that 
\begin{align}
&\int_\Gamma f(y) \int_0^\infty \mathbb{P}_0^\triangleleft(X_t \in\d y)\d t\notag\\
&=\int_0^\infty \d t  \int_\Gamma f(y) M(y)n_t(y)\d y\notag\\
&=\int_0^{\infty} {\rm d}t\int_\Gamma f(y) M(y) \lim_{x\to 0} \frac{p_t^{\Gamma}(x,y)}{\mathbb{P}_x(\kappa_{\Gamma}>1)}{\rm d} y\notag\\
&= \int_0^{1} {\rm d}t \int_\Gamma f(y) M(y) \lim_{x\to 0} \frac{p_t^{\Gamma}(x,y)}{\mathbb{P}_x(\kappa_{\Gamma}>1)}{\rm d} y+\int_1^{\infty} {\rm d}t  \int_\Gamma f(y) M(y) \lim_{x\to 0} \frac{p_t^{\Gamma}(x,y)}{\mathbb{P}_x(\kappa_{\Gamma}>1)}{\rm d} y\notag\\
&=\int_0^{1} {\rm d}t \int_\Gamma f(y) M(y) \lim_{x\to 0} \frac{ p_t^{\Gamma}(x,y)}{\mathbb{P}_x(\kappa_{\Gamma}>1)}{\rm d} y\notag\\
&\hspace{2cm}+ \int_1^{\infty} {\rm d}t \int_\Gamma f(y) M(y)\lim_{x\to 0} \frac{t^{-d/\alpha} p_1^{\Gamma}(t^{-1/\alpha}x,t^{-1/\alpha}y)}{\mathbb{P}_x(\kappa_{\Gamma}>1)}{\rm d} y
\label{2terms}
\end{align}
We wish to use Dominated Convergence theorem to pull the limit out of  each of the integrals. Referring again to Theorem \ref{entrance}, and recalling the compactness of the support of  $f$, the integrand in the first term on the righthand side of \eqref{2terms} is uniformly  bounded.

\smallskip

For the second term on the right-hand side of \eqref{2terms}, we can assume without loss of generality that the support of $f$ lies in $\Gamma\cap\{x\in\mathbb{R}^d: |x|<1\}$. Recall again from   the bound in Lemma 4.2 of \cite{BB}, which states that, for $t>1$ and $|x|<1$, there exists a constant $C>0$ such that 
\[
C^{-1} t^{-\beta/\alpha}M(x)<\mathbb{P}_{t^{-1/\alpha}x}(\kappa_{\Gamma}>1)< C t^{-\beta/\alpha}M(x).
\]

Using the above, and appealing in particular  to equation (4.16) of Bogdan et al. \cite{BPW}, for $t,|x|>1$ and $y\in\Gamma$,
\begin{align*}
\frac{t^{-d/\alpha} p_1^{\Gamma}(t^{-1/\alpha}x,t^{-1/\alpha}y)}{\mathbb{P}_x(\kappa_{\Gamma}>1)}
&<\frac{t^{-d/\alpha} p_1^{\Gamma}(t^{-1/\alpha}x,t^{-1/\alpha}y)}{M(x)}
\\&<t^{-(d+\beta)/\alpha}C \frac{\mathbb{P}_{t^{-1/\alpha}y}(\kappa_{\Gamma}>1)}{(1+ t^{-1/\alpha}|y|)^{d+\alpha}}\\
&<t^{-(d+2\beta)/\alpha} .
\end{align*}
The right-hand side above can now be used as part of a dominated convergence argument for the second term in \eqref{2terms}, noting in particular that $f$ is compactly supported.

\smallskip

In conclusion, we have
\begin{align}
&\int_\Gamma f(y) \int_0^\infty \mathbb{P}_0^\triangleleft(X_t \in\d y)\d t\notag\\
&= \lim_{x\to 0} \int_\Gamma f(y) M(y) \int_0^{1} {\rm d}t\frac{ p_t^{\Gamma}(x,y)}{\mathbb{P}_x(\kappa_{\Gamma}>1)}{\rm d} y\notag\\
&\hspace{2cm}+ \lim_{x\to 0}  \int_\Gamma f(y) M(y)\int_1^{\infty} {\rm d}t\frac{ p_t^{\Gamma}(x,y)}{\mathbb{P}_x(\kappa_{\Gamma}>1)}{\rm d} y\notag\\
&=
\lim_{x\to 0}\frac{\int_\Gamma f(y)M(y) G^{\Gamma}(x,y)}{M(x)}.
\label{weakconvergence}
\end{align}
The above limit has already been computed in Lemma 3.5 of Bogdan et al. \cite{BPW} and agrees with the conclusion of this Lemma. 
\end{proof}

To complete the proof of part (ii) of Theorem \ref{BZcone}, we must show \eqref{switching}. To this end, let us start by recalling Hunt's switching identity for $X$ as a symmetric process and $\kappa_\Gamma$ as a hitting time of an open domain. It ensures that for any $f,g:\Gamma\to\mathbb{R}$ measurable and bounded one has
\[
\int_{\Gamma}\d x f(x)\int_{\Gamma}\d y \,p^\Gamma_t(x, y)g(y) = \int_{\Gamma}\d x g(x)\int_{\Gamma}\d y \,p^\Gamma_t(x, y)f(y),\qquad \forall t\geq 0.
\]
With this in hand, it is easy to check that 
\begin{align*}
\int_{\Gamma}\mu(\d x)f(x)\int_{\Gamma}\d y \,p^\triangleleft_t(x, y)g(y) & =
\int_{\Gamma}\d xf(x)M(x)H(x)\int_{\Gamma}p^\Gamma_t(x, y)g(y)\frac{M(y)}{M(x)}\\
&=\int_{\Gamma}\d x g(x)M(x)H(x)\int_{\Gamma}p^\Gamma_t(x, y)f(y)\frac{H(y)}{H(x)}\\
&=\int_{\Gamma}\mu(\d x)g(x)\int_{\Gamma}\d y \,p^\triangleright_t(x, y)f(y),\qquad  t\geq 0,
\end{align*}
as required.
\hfill$\square$

\section{Proof of Theorem \ref{0}}\label{proof0}

To prove the weak convergence on the Skorokhod space of $\mathbb{P}_x$, as $x\to0$, to $\mathbb{P}_0$, we appeal to the following proposition, lifted from Dereich et al. \cite{DDK} and written in the language of the present context.

\begin{proposition}\label{prop}
	Recall $\tau^\ominus_{\eps}=\inf\{t: |X_t|\geq \eps\}$, $\eps>0$. Suppose that the following  conditions hold: 
\begin{enumerate}
\item[(a)] $\lim_{\eps\to 0} \limsup_{\Gamma\ni z \to 0}{\mathbb E}^{\triangleleft}_z[\tau^\ominus_{\eps}]=0$
\item[(b)] $\lim_{\Gamma\ni z\to 0} {\mathbb P}^\triangleleft_{z}(X_{\tau^\ominus_{\eps}}\in\cdot)=:\mu_\eps(\cdot)$ exists for all $\eps>0$
\item[(c)] ${\mathbb P}^{\triangleleft}_0$-almost surely, $X_0=0$ and $X_t\not=0$ for all $t>0$
\item[(d)] ${\mathbb P}^{\triangleleft}_0((X_{\tau^\ominus_{\eps}+t})_{t\geq 0}\in\cdot)=\int_\Gamma \mu_\eps(\d y){\mathbb P}^\triangleleft_{y}(\cdot)$ for every $\eps>0$
\end{enumerate}
Then the mapping
$$
	\Gamma\ni z\mapsto {\mathbb P}^\triangleleft_z
$$
is continuous in the weak topology on the Skorokhod space.
\end{proposition}

{\it Verification of Condition (a).}
Define $G^\triangleleft(x,y)$ via the relation 
\[
\int_\Gamma f(y)G^\triangleleft(x,y)\d y = \mathbb{E}^\triangleleft_x\left[\int_0^\infty f(X_t)\d t\right], 
\]
and note that $G^\triangleleft(x,y) = M(y)G^\Gamma(x,y)/M(x)$, $x,y\in\Gamma$.
Then, for $f$ positive, bounded, measurable and compactly supported and $x\in\Gamma\cup\{0\}$, then \eqref{weakconvergence} and Lemma \ref{excursionpotential} tells us that 
\[
\lim_{\Gamma\ni z\to0}\int_\Gamma f(y)G^\triangleleft(z,y)\d y = \int_\Gamma f(y)G^\triangleleft(0,y)\d y.
\]
Now note that 
\begin{align*}\lim_{\varepsilon\to 0} \limsup_{\Gamma\ni z \to 0}{\mathbb E}^{\triangleleft}_z[\tau^\ominus_{\varepsilon}]&=\lim_{\varepsilon\to 0} \limsup_{\Gamma\ni z \to 0}{\mathbb E}^{\triangleleft}_z\left[\int_0^{\tau^\ominus_{\varepsilon}}\mathbf{1}_{(|X_t| <\varepsilon)}{\rm d} t\right]\\
&\leq\lim_{\varepsilon\to 0} \limsup_{\Gamma\ni z \to 0}{\mathbb E}^{\triangleleft}_z\left[\int_0^{\infty}\mathbf{1}_{(|X_t| <\eps)}{\rm d} t\right]\\
&\leq  \lim_{\varepsilon\to 0} \limsup_{\Gamma\ni z \to 0}\int_{|y|<\varepsilon}G^{\triangleleft}(z, y){\rm d}y\\
&\leq C \lim_{\varepsilon\to 0} \int_{|y|<\varepsilon}H(y)M(y){\rm d}y \\
&\leq C \lim_{\varepsilon\to 0} \int_{\Omega} \sigma_1({\rm d} \theta)M(\theta)^2\int_0^\varepsilon r^{\alpha-\beta-1}{\rm d}r\\
&\leq C \lim_{\varepsilon\to 0} \varepsilon^{\alpha-\beta}\\
&=0,
\end{align*}
where $C\in(0,\infty)$ is an unimportant constant which changes its value in each line and $\sigma_1(\d \theta)$ is the surface measure on $\mathbb{S}^{d-1}$ normalised to have unit mass.

\smallskip

{\it Verification of Condition (b)}. This condition is covered by Theorem \ref{oo}. Note, moreover, that $\mu_\varepsilon$ does not depend on $\varepsilon$.

\smallskip

{\it Verification of  Condition (c)}. This condition is covered by Theorem \ref{BZcone}.

\smallskip

{\it Verification of Condition (d)}.  We have  that, for $|x|<\eta<\varepsilon$, from the Strong Markov Property,
\begin{equation}
\mathbb{E}^\triangleleft_x[f((X_{\tau^\ominus_{\eps}+t}: t\geq 0)]=
\mathbb{E}^\triangleleft_x\left[ \mathbb{E}_{X_{\tau^\ominus_{\eps}}}[f(X_t : t\geq 0)] \right] = 
\mathbb{E}^\triangleleft_x\left[ g(X_{\tau^\ominus_{\eta}}  ) \right] 
\label{useDCTagain}
\end{equation}
for bounded, measurable $f$, where 
\[
g(y) = \mathbb{E}^\triangleleft_y\left[\mathbb{E}^\triangleleft_{X_{\tau^\ominus_{\eps}}}[f(X_t : t\geq 0)].
\right]
\]
is bounded and measurable. From Theorem \ref{oo} and the Skorokhod continuity of $X, \mathbb{P}^\triangleleft$, which follows from the Lamperti--Kiu representation \eqref{eq:lamperti_kiu}, we can take limits in \eqref{useDCTagain} to get 
\begin{equation}
\mathbb{E}^\triangleleft_0[f((X_{\tau^\ominus_{\eps}+t}: t\geq 0)]=
\mathbb{E}^\triangleleft_0\left[ \mathbb{E}_{X_{\tau^\ominus_{\eps}}}[f(X_t : t\geq 0)] \right] = 
\mathbb{E}^\triangleleft_0\left[ g(X_{\tau^\ominus_{\eta}}  ) \right] 
\label{useDCTagain2}
\end{equation}
Now appealing to Theorem \ref{BZcone} (ii), thanks to c\`adl\`ag paths, we know that $X_{\tau^\ominus_{\eta}} \to0$ almost surely under $\mathbb{P}_0^\triangleleft$. As a consequence, we can appeal to the Dominated Convergence Theorem in \eqref{useDCTagain2}, together with condition (a) and, again, the Skorokhod continuity of $X$ under $\mathbb{P}^\triangleleft_y$, $y\in\Gamma$, and deduce the statement in condition (d).

\section{It\^o synthesis and proof of Theorem \ref{recex}}\label{proofrecex}

The basis of Theorem \ref{recex} is the classical method of It\^o synthesis of Markov processes and an extension of the main ideas in \cite{R05}. That is to say, the technique of piecing together excursions end to end under appropriate conditions, whilst ensuring that  the strong Markov property holds. In our case, we are also charged with ensuring that self-similarity is preserved as well. We split the proof of Theorem \ref{recex} into the construction of the recurrent extension and the existence and characterisation of a stationary distribution.

\subsection{Some general facts on self-similar recurrent extensions}
As described before the statement of Theorem~\ref{recex}, according to It\^o's synthesis theory a self-similar recurrent extension of $(X, \mathbb{P}^{\Gamma})$ can be build from a self-similar excursion measure, i.e. a measure on $\mathbb{D}$ satisfying the conditions (i)-(iv) stated just before Theorem~\ref{recex}. 

\smallskip

Suppose that  $\mathbf{N}^{\Gamma}$ is a self-similar excursion measure compatible with the semigroup of $(X, \mathbb{P}^{\Gamma})$. Define a Poisson point process $((s,\chi_s), s>0)$ on $(0,\infty)\times\mathbb{D}$ with intensity $\d t\times \mathbf{N}^{\Gamma}(\d \chi)$ and let each excursion length be denoted by 
$$\zeta_s :=\inf\{t>0:\chi_s(t)=0\} >0 .$$
Then, via the subordinator 
\[
\varsigma_t = \sum_{s\leq t} \zeta_s,\qquad  t\geq 0 ,
\]
 we can define a local time process at $0$ by 
\[
L_t = \inf\{r>0:\varsigma_r>t\}, \qquad t\geq 0
\]
Note, for each $t\geq 0$, by considering the Laplace transform of  $\varsigma_t$, Campbell's formula and the assumption that $\mathbf{N}^{\Gamma}(1-{\rm e}^{-\zeta})<\infty$ ensures that $(\varsigma_t, t\geq 0)$ is well defined as a subordinator with  jump measure given by 
$\nu(\d s) = \mathbf{N}^{\Gamma}(\zeta\in \d s) $, $s>0$.

\smallskip


\smallskip

Now, we define $({\cenX}_t,t\geq 0)$ with the following pathwise construction. For $t\geq 0$, let $L_t=s,$ then $\varsigma_{s-}\leq t\leq \varsigma_{s-}$ and define 
$${\cenX}_t := \begin{cases} 
\Delta_s(t-\varsigma_{s-}), &\text{ if }\varsigma_{s-}<\varsigma_s,\\ 
0, &\text{ if } \varsigma_{s-}=\varsigma_s \text{ or }s=0.
\end{cases}
$$
Salisbury \cite{Salt1986a, Salt1986b} demonstrates how the process constructed above preserves the Markov property. In fact, one can easily adapt the arguments provided by Blumenthal \cite{Blum1983}, who considers only $[0,\infty)$ valued processes, to show that, under some regularity hypotheses on the semigroup of the minimal process $(X,\mathbb{P}^{\Gamma}),$ the process constructed above is a Feller process. This is due to the fact that, here we are considering an extension from $\Gamma$ to $\Gamma\cup\{0\},$ for $(X,\mathbb{P}^{\Gamma})$, which, by decomposing this process into polar coordinates, is equivalent to extend the radial part from $(0,\infty)$ to $[0,\infty)$.  


\smallskip

To thus verify the Feller property, suppose that $C_0(\Gamma)$ is the space of continuous functions on $\Gamma$ vanishing at 0 and $\infty,$ and we write $({\mathcal P}^\Gamma_t, t\geq 0)$ for the semigroup of $(X,\mathbb{P}^\Gamma)$. The aforesaid regularity hypothesis needed to adapt the argument given by Blumenthal \cite{Blum1983} are:
 \begin{description}
\item[(i)] If $f\in C_0(\Gamma)$, then ${\mathcal P}^\Gamma_t f\in C_0(\Gamma)$ and ${\mathcal P}^\Gamma_t f \mapsto f$ uniformly as $t\to0$;
\item[(ii)] For each $q>0$, the mapping $x\mapsto\mathbb{E}^\Gamma_x[{\rm e}^{-q\zeta}]$ is continuous in $\Gamma$;
\item[(iii)] The following limits hold;
\[
 \lim_{\Gamma\ni x\to0} \mathbb{E}^\Gamma_x[{\rm e}^{-\zeta} ] = 1 \text{ and } \lim_{x\in\Gamma,\, |x|\to\infty} \mathbb{E}^\Gamma_x[{\rm e}^{- \zeta} ] =0.
\]
\end{description}
All of these are easily verified using the Lamperti--Kiu representation of $(X,\mathbb{P}^\Gamma)$ .

\smallskip

Now that we know that the process $({\cenX}_t,t\geq 0)$ defined above is a strong Markov process, in fact a Feller process, we should verify that such a process has the scaling property. But this is a consequence of the condition (iv) above, as can be easily verified using the arguments in the proof of Lemma 2 in \cite{R05}.   
\smallskip

We will next describe all the excursion measures $\mathbf{N}^{\Gamma},$ compatible with $(X, \mathbb{P}^\Gamma).$ To that end, we recall that the entrance law  $(\mathbf{N}^{\Gamma}_t(\d y), t>0)$ of an excursion measure $\mathbf{N}^{\Gamma},$ is defined by  
$$\mathbf{N}^{\Gamma}_t(\d y):=\mathbf{N}^{\Gamma}(X_t\in \d y, t<\zeta),\qquad t>0.$$ 

\begin{lemma}\label{lemmadecomposition} Let $\mathbf{N}^{\Gamma}$ be a self-similar excursion measure compatible with $(X, \mathbb{P}^{\Gamma})$, and $\gamma$ the index appearing in (iv). Then, its entrance law admits the following representation:  there is a constant $a\geq 0$, such that for all $t>0$ and any $f:\Gamma\mapsto \mathbb{R}^+$ continuous and bounded
\begin{equation}\label{EL-decomp}
\begin{split}
\mathbf{N}^{\Gamma}(f(X_t),t<\zeta) &=a \lim_{|x|\to 0} \frac{\mathbb{E}_x[f(X_t),t<\kappa_\Gamma]}{M(x)}\\
&\qquad +\int_{|y|>0} \mathbf{N}^{\Gamma}(X_{0+}\in{\rm d}y)\mathbb{E}_{y}[f(X_t),t<\kappa_\Gamma].
\end{split}
\end{equation}
 Furthermore, there is a measure $\pi^{\Gamma}$ on $\Omega$ such that  
\begin{equation}\label{rep.X0}
\mathbf{N}^{\Gamma}(|X_{0+}|\in{\rm d}r, \arg(X_{0+})\in{\rm d}\theta )=\frac{{\rm d} r}{r^{1+\alpha \gamma}}\pi^{\Gamma}({\rm d}\theta),
\end{equation}
and $\int_{\Omega}\pi^{\Gamma}({\rm d}\theta)M(\theta)<\infty.$ Finally, necessarily  $\gamma \in(0,1)$, and $\gamma \leq \beta/\alpha;$ if $\gamma=\beta/\alpha$ then the measure $\pi^{\Gamma}\equiv 0,$ whilst if $\gamma<\beta/\alpha$, then $a\equiv 0.$ 
\end{lemma} 
\begin{proof} In order to prove the decomposition (\ref{EL-decomp}) we start by noticing that for all $s,t>0,$ we have 
$$\mathbf{N}^{\Gamma}(f(X_t),t<\zeta)= \mathbf{N}^{\Gamma}(\lim_{s\to 0}f(X_{s+t}),s+t<\zeta),$$ which is a consequence of the dominated convergence theorem, since 
$$ \mathbf{N}^{\Gamma}(f(X_{s+t}),s+t<\zeta) \leq ||f|| \mathbf{N}^{\Gamma}(s+t<\zeta)\leq ||f||\mathbf{N}^{\Gamma}(t<\zeta)<\infty,$$ since $\mathbf{N}^{\Gamma}(t<\zeta)$ is always finite for any $t>0$ because 
$$\infty>\mathbf{N}^{\Gamma}(1-{\rm e}^{-\zeta})>\mathbf{N}^{\Gamma}(1-{\rm e}^{-\zeta},t<\zeta)>(1-{\rm e}^{-t})\mathbf{N}^{\Gamma}(t<\zeta).$$
The former, together with the Markov property under $\mathbf{N}^{\Gamma},$ implies that
\begin{align*}
\mathbf{N}^{\Gamma}(f(X_t), t<\zeta)&= \mathbf{N}^{\Gamma}(\lim_{s\to 0}f(X_{s+t}),s+t<\zeta)\\
&= \lim_{s\to 0}\mathbf{N}^{\Gamma}(f(X_{s+t}),s+t<\zeta)\\
&=\lim_{s\to 0}\mathbf{N}^{\Gamma}(\mathbb{E}_{X_s}[f(X_t),t<\kappa_\Gamma],s<\zeta)\\
&= \lim_{\epsilon\to 0}\lim_{s\to 0}\mathbf{N}^{\Gamma}(\mathbb{E}_{X_s}[f(X_t),t<\kappa_\Gamma],|X_s|<\epsilon,s<\zeta)\\
&\qquad +\lim_{\epsilon\to 0}\lim_{s\to 0}\mathbf{N}^{\Gamma}(\mathbb{E}_{X_s}[f(X_t),t<\kappa_\Gamma],|X_s|>\epsilon,s<\zeta)\\
&= \lim_{\epsilon\to 0}\lim_{s\to 0}\mathbf{N}^{\Gamma}\left(M(X_s)\frac{\mathbb{E}_{X_s}[f(X_t),t<\kappa_\Gamma]}{M(X_s)},|X_s|<\epsilon,s<\zeta\right)\\
&\qquad +\int_{y\in \Gamma, |y|>0} \mathbf{N}^{\Gamma}(X_{0+}\in{\rm d}y)\mathbb{E}_{y}[f(X_t),t<\kappa_\Gamma];\end{align*}
where in the final equality we used the continuity of the mapping $y\mapsto\mathbb{E}_{y}[f(X_t),t<\kappa_\Gamma],$ $y\in\Gamma.$
Corollary 3.2 and Theorem 3.3 in   \cite{BPW} implies that the limit 
$\lim_{|x|\to 0} \mathbb{E}_x[f(X_t),t<\kappa_\Gamma]/M(x)$ exists.
This implies from   the right hand side above that
\begin{align*}\mathbf{N}^{\Gamma}(f(X_t), t<\zeta)&= a\lim_{|x|\to 0}\frac{\mathbb{E}_{x}[f(X_t),t<\kappa_\Gamma]}{M(x)}+\int_{y\in \Gamma, |y|>0} \mathbf{N}^{\Gamma}(X_{0+}\in\d y)\mathbb{E}_{y}[f(X_t),t<\kappa_\Gamma],
\end{align*}
where
\begin{align*}& a = \lim_{\epsilon\to 0}\lim_{s\to 0}\mathbf{N}^{\Gamma}(M(X_s), |X_s|<\epsilon, s<\zeta)\\
&= \lim_{\epsilon\to 0}\lim_{s\to 0}\mathbf{N}^{\Gamma}\left(\frac{M(X_s)}{\mathbb{P}_{X_s}(\kappa_\Gamma>1)}\mathbb{P}_{X_s}(\kappa_\Gamma>1),|X_s|<\epsilon,s<\zeta\right)\\
&=\left(\lim_{|x|\to 0}\frac{M(x)}{\mathbb{P}_{x}(\kappa_\Gamma>1)}\right)\lim_{\epsilon\to 0}\lim_{s\to 0}\mathbf{N}^{\Gamma}(|X_s|<\epsilon,1+s<\zeta) <\infty.
\end{align*}
This finishes the proof of the identity (\ref{EL-decomp}). We will next prove the identity (\ref{rep.X0}). The latter decomposition together with the convergence (\ref{P0}), applied to $f(x)=M(x)g(\arg(x))\mathbf{1}_{(|x|\in (0,1))}$ with $g$ any continuous and bounded function on $\Gamma,$ implies that 
\begin{equation*}
\begin{split}
&\mathbf{N}^{\Gamma}\left(M(X_t)g(\arg(X_t))\mathbf{1}_{(|X_t|\in (0,1)}, t<\zeta\right)\\
&=a C \int_{|y|<1}g(\arg(y))n_t(y)\d y\\
&\hspace{2cm}+\int_{y\in \Gamma, |y|>0} \mathbf{N}^{\Gamma}(X_{0+}\in\d y)\mathbb{E}_{y}[M(X_t)g(\arg(X_t))\mathbf{1}_{(|X_t|\in(0,1))}, t<\kappa_\Gamma]
\end{split}
\end{equation*}
By the scaling property (iv') applied to $f(x)=M(x)g(\arg(x))\mathbf{1}_{( |x|\in (0,1) )}.$ 
\begin{equation*}
\begin{split}
\mathbf{N}^{\Gamma}(M(X_{0+})g(\arg(X_{0+})), |X_{0+}|\in (0,1)))&=c^{\alpha \gamma} \mathbf{N}^{\Gamma}(M(c^{-1}X_{0+})g(\arg(X_{0+})), |X_{0+}|\in (0,c))\\
&=c^{\alpha \gamma-\beta} \mathbf{N}^{\Gamma}(M(X_{0+})g(\arg(X_{0+})), |X_{0+}|\in (0,c)).\end{split}\end{equation*}
Notice that this is always finite because for $|x|<1$, $M(x)< K\mathbb{P}_x(\kappa_\Gamma>1)$, for some $K>0$. So, by the Markov property, the latter is bounded by $c^{\alpha\gamma-\beta}K||g|| \mathbf{N}^{\Gamma}(\zeta>1)$. Differentiating in $c>0,$ one gets 
\begin{equation*}
\begin{split}
(\beta-\alpha \gamma)r^{\beta-\alpha\gamma-1}{\rm d} r\, &\mathbf{N}^{\Gamma}(M(X_{0+})g(\arg(X_{0+})) ,|X_{0+}|\in (0,1))\\
&= \mathbf{N}^{\Gamma}(M(X_{0+})g(\arg(X_{0+})), |X_{0+}|\in {\rm d} r).
\end{split}
\end{equation*}
Observe that since the right hand side is positive as soon as $g$ is positive, we get as a side consequence that $\beta-\alpha\gamma\geq 0.$ Using again that $M(x)=|x|^{\beta}M(\arg(x)),$ one gets
\begin{equation*}
\begin{split}
(\beta-\alpha \gamma)r^{-\alpha\gamma-1}{\rm d} r\, &\mathbf{N}^{\Gamma}(|X_{0+}|^{\beta}M(\arg(X_{0+}))g(\arg(X_{0+})) ,|X_{0+}|\in (0,1))\\
&= \mathbf{N}^{\Gamma}(M(\arg(X_{0+}))g(\arg(X_{0+})), |X_{0+}|\in {\rm d} r).
\end{split}
\end{equation*}
Since this identity holds for any $g$ continuous and bounded, we derive that, when $\beta>\alpha\gamma,$ the equality of measures
$$\mathbf{N}^{\Gamma}(|X_{0+}|\in {\rm d} r,\ \arg(X_{0+})\in {\rm d} \theta))=\frac{{\rm d} r}{r^{\alpha \gamma+1}} \pi^{\Gamma}({\rm d}\theta),$$ holds, where,
$$ \pi^{\Gamma}({\rm d}\theta)=\frac{\mathbf{N}^{\Gamma}(|X_{0+}|^{\beta}, |X_{0+}|\in (0,1), \arg(X_{0+})\in  {\rm d} \theta))}{\beta-\alpha \gamma}. $$
Whilst if $\beta=\alpha\gamma,$ $\mathbf{N}^{\Gamma}(|X_{0+}|>0)\equiv 0,$ and thus $\pi^{\Gamma}\equiv 0.$ We are just left to prove that when $\beta>\alpha\gamma,$ then $\pi^{\Gamma}M<\infty$ and $a\equiv 0.$ Indeed, that $\pi^{\Gamma}M<\infty$ follows from the following estimates 
\begin{align*}
\infty&>n(1-{\rm e}^{\zeta},X_{0+}\neq 0)\\
&=\int_0^{\infty} {\rm d} s \, n(s<\zeta,X_{0+}\neq 0){\rm e}^{-s}\\
&=\int_{0}^\infty{\rm d} s \,{\rm e}^{-s}\int_{\Omega}\pi^{\Gamma}({\rm d}\theta) \int_0^\infty\frac{{\rm d} r}{r^{1+\alpha \gamma}}\mathbb{P}_{r\theta}(\kappa_\Gamma>s)  \\
&\geq \int_{\Omega}\pi^{\Gamma}({\rm d}\theta) \int_0^1\frac{{\rm d} r}{r^{1+\alpha \gamma}}\int_{r^{\alpha}}^\infty {\rm d} s {\rm e}^{-s}M(r\theta)s^{-\beta/\alpha}\\
&\geq\int_{\Omega}\pi^{\Gamma}({\rm d}\theta)M(\theta) \int_0^1\frac{{\rm d} r}{r^{1+\alpha \gamma-\beta}}\int_{1}^\infty{\rm d} s {\rm e}^{-s}s^{-\beta/\alpha},
\end{align*}
where we used the estimate in Proposition \ref{BPWprop}. As claimed we derive that $\int_{\Omega}\pi^{\Gamma}({\rm d} \theta)M(\theta)<\infty$. To finish, we observe that the identity \eqref{EL-decomp} together with the scaling property (iv') implies that for any $t>0$
\begin{equation*}
\begin{split}
t^{-\gamma}\mathbf{N}^{\Gamma}(\zeta>1)=\mathbf{N}^{\Gamma}(\zeta>t)\geq a \lim_{|x|\to 0}\frac{\mathbb{P}_{x}(t<\kappa_\Gamma)}{M(x)}=a t^{-\beta/\alpha}C,
\end{split}
\end{equation*}
with $C>0$ the constant appearing in Proposition \ref{BPWprop}. Since by assumption $\beta>\alpha\gamma,$ we obtain by making $t\downarrow 0,$ that $a\equiv 0.$ We have thus finished the proof of Lemma~\ref{lemmadecomposition}.
\end{proof}
To finish the proof of Theorem~\ref{recex} one should notice that, from the equation~\eqref{P0}, for every $f:\Gamma\to \mathbb{R}^+$ continuous with compact support, one has the convergence 
\begin{equation*}
\begin{split}
\lim_{|x|\to 0} \frac{\mathbb{E}_{x}[f(X_t),t<\kappa_\Gamma]}{M(x)}=C\int_{\Gamma}f(y)n_t(y)\d y, \qquad t>0.
\end{split}
\end{equation*}
This together with the decomposition in Lemma~\ref{lemmadecomposition} implies that we necessarily have the following representation for the entrance law of any self-similiar excursion measure $\mathbf{N}^{\Gamma}$. There is a measure $\pi^{\Gamma}$ on $\Omega,$ such that $\int_{\Omega}\pi^{\Gamma}(\d\theta)M(\theta)<\infty,$ and a constant $a\geq 0$ such that
\begin{equation}\label{EL-decomp2}
\begin{split}
\mathbf{N}^{\Gamma}(X_t\in\d y, t<\zeta) =a n_t(y)\d y+\int_0^\infty\frac{\d r}{r^{1+\alpha\gamma}}\int_{\Omega}\pi^{\Gamma}(\d\theta)\mathbb{E}_{r\theta}[X_t \in \d y,\ t<\kappa_\Gamma],
\end{split}
\end{equation}  
$a\pi^{\Gamma}\equiv 0,$ if $a>0$ then $\gamma=\beta/\alpha,$ and if $\pi^{\Gamma}\neq 0$ then $\gamma<\beta/\alpha,$ and $\pi^{\Gamma}M<\infty.$ 


Furthermore, via cylinder sets, one can check that for $t>0$ and $A\in\mathcal{F}_t$ 
\begin{equation}\label{EL-decomp3}
\begin{split}
\mathbf{N}^{\Gamma}(A \,,\,  t<\zeta) =a\mathbb{E}^{\triangleleft}_0\left[\frac{1}{M(X_t)}\mathbf{1}_A\right] +\int_0^\infty\frac{\d r}{r^{1+\alpha\gamma}}\int_{\Omega}\pi^{\Gamma}(\d\theta)\mathbb{E}_{r\theta}[A \, ,\, t<\kappa_\Gamma].
\end{split}
\end{equation}  
 with $a$ and $\pi^{\Gamma}$ as above. As a side consequence, we have that the measure $\tilde{\mathbf{N}}^{\Gamma}$ on $\mathbb{D},$ defined by the relation 
\begin{equation}
\tilde{\mathbf{N}}^{\Gamma}(A \,,\, t<\zeta) := \mathbb{E}^{\triangleleft}_0\left[\frac{1}{M(X_t)}\mathbf{1}_A\right],\quad  \text {for\ }  A\in\mathcal{F}_t, \  t>0,  
\end{equation}
is a self-similar excursion measure, whose entrance law is $(n_t, t>0),$ and such that $\tilde{\mathbf{N}}^{\Gamma}(X_{0+}\neq 0)=0.$ Finally this is the unique self-similar excursion measure bearing this property, and hence the self-similar recurrent extension associated to it leaves zero continuously, and it is the unique self-similar recurrent extension having this property.

\subsection{Invariant measure}
We start by computing the invariant measure according to Chapter XIX.46 of Dellacherie and Meyer \cite{DM}. There it is shown that the invariant measure $\cenpi\!{}^\Gamma (\d y)$, $y\in \Gamma$, defined up to a multiplicative constant, is given by the excursion occupation measure so that
\[
\int_\Gamma f(y)\cenpi\!{}^\Gamma (\d y) = \mathbf{N}^{\Gamma}\left(\int_0^\zeta f(\chi_t)\d t\right),
\]
for all  bounded measurable $f$ on $\Gamma$. Note, however, the computations in Lemma \ref{excursionpotential} can be used to show that 
\[
\tilde{\mathbf{N}}^{\Gamma}\left(\int_0^\zeta f(\chi_t)\d t\right) = \int_0^\infty\int_\Gamma f(z) n_t(z)\d z\,\d t = \int_\Gamma \frac{f(z)}{M(z)}G^\triangleleft(0,\d z)
=\int_\Gamma f(z)H(z)\d z.
\]
It is then straight forward to prove the final identity in the statement of Theorem~\ref{recex}.

\medskip

Finally, to see that $\cenpi$ is not a finite measure, we can compute its total mass, after converting to generalised polar coordinates (see e.g. Blumenson \cite{blum}), by
\begin{equation}
\int_\Gamma H(x)\d x = C\int_\Omega\sigma_1(\d \theta)M(\theta)\int_0^\infty r^{\alpha-\beta -1}\d r = \infty,
\label{totalmass}
\end{equation}
where $C>0$ is an unimportant constant attached to the Jacobian in the change of variables to generalised polar coordinates, and $\sigma_1(\d\theta)$ is the surface measure on $\mathbb{S}^{d-1}$ normalised to have unit total mass. Moreover, we also have that if $\pi^{\Gamma}$ is not trivial then
$$\int_0^\infty\frac{\d r}{r^{1+\alpha\gamma}}\int_{\Omega}\pi^{\Gamma}(\d\theta)\mathbb{E}_{r\theta}\left[\kappa_\Gamma\right]=\infty,$$ because by Proposition~\ref{BPWprop}, $\mathbb{E}_{r\theta}\left[\kappa_\Gamma\right]=\infty,$ for any $r>0,$ and $\Omega.$

\smallskip

The failure of this measure to normalise to have unit mass means that a stationary distribution cannot exist, cf. Chapter XIX.46 of Dellacherie and Meyer \cite{DM} and hence $\cenX$ is a null-recurrent process. 
\hfill$\square$

\section{Proof of Theorem  \ref{recexctsat0}}\label{proofrecexctsat0} The proof of this result follows verbatim that of Theorem \ref{recex}, albeit for some of the estimates that are used. Indeed, to establish Theorem \ref{recex}, we used that for $(X,\mathbb{P}^{\Gamma})$ we have
\begin{description}
\item[(a)] for $t>|x|^{\alpha}$,
$$ \mathbb{P}_x(\kappa_\Gamma>t) \approx M(x) t^{-\beta/\alpha};$$
\item[(b)] for any $t>0$ and $f:\Gamma\to \mathbb{R}^+$ continuous and bounded
$$\lim_{|x|\to 0} \frac{\mathbb{E}_x(f(X_t),t<\kappa_\Gamma)}{M(x)} \text{ exists.}$$
\end{description}
These conditions are replaced by the following conditions on  $(X, \mathbb{P}^{\triangleright})$ 
\begin{description}
\item[(a')] for $t>|x|^{\alpha}$,
$$ \mathbb{P}^{\triangleright}_x(\kappa_\Gamma>t) \approx H(x) t^{(\alpha-d-2\beta)/\alpha},$$
\item[(b')] for any $t>0$ and $f:\Gamma\to \mathbb{R}^+$  continuous and bounded
$$\lim_{|x|\to 0} \frac{\mathbb{E}^{\triangleright}_x(f(X_t),t<\kappa_\Gamma)}{H(x)} \text{ exists.}$$
\end{description}
Moreover, in proving Theorem \ref{recexctsat0}, where one reads $\beta$ in the proof of Theorem \ref{recex}, one should use $\beta^{\prime}:=2\beta+d-\alpha.$ From here we have the restriction $0<\beta^{\prime}=d+2\beta-\alpha<\alpha$, which restricts $\beta$ to the interval $(\left({\alpha-d}\right)/2)\vee 0<\beta<(2\alpha-{d})/{2}.$    

\smallskip

Let us finish by noticing that the finiteness of $\mathbf{N}^\triangleright(1-{\rm e}^{-\zeta})$ is equivalent to  $d<2(\alpha-\beta)$. To this end, we can appeal to Lemma 4.3 of \cite{BB}  to see that, there exists a constant $C>0$, such that, for $x\in\Gamma$,
\begin{equation}
C^{-1} s^{(\alpha-2\beta-d)/\alpha}<\lim_{x\to 0} |x|^{\alpha-2\beta-d}\mathbb{P}^{\triangleright}_x (s<\kappa_{\{0\}}) <C s^{(\alpha-2\beta-d)/\alpha},
\label{2sidedbound}
\end{equation}
where $\kappa_{\{0\}} = \inf\{t>0 : |X_t| = 0\}$.
Recall from \eqref{COMH2} that, for $x\in\Gamma$, and $A\in\mathcal{F}_t$, $t\geq 0$, 
\begin{equation}
\mathbb{E}^\triangleleft_x\left[
\mathbf{1}_{A}\frac{|X_t|^{\alpha -2\beta-d}}{|x|^{\alpha - 2\beta-d}}
\right] = 
  \mathbb{E}_x\left[\frac{H(X_t)}{H(x)}\mathbf{1}_{(A\cap \{t<\kappa_\Gamma\}) }\right] =  \mathbb{P}_x^\triangleright (A,\, t<\kappa_{\{0\}}) .
  \label{leftotright}
\end{equation}

Hence, using \eqref{leftotright}, we have 
\begin{align*}
\mathbf{N}^\triangleright(1-{\rm e}^{-\zeta})&=\int_0^{\infty} {\rm e}^{-s}\mathbf{N}^{\triangleright}(\zeta > s){\rm d} s\\
&= \int_0^{\infty}{\rm d}s\, {\rm e}^{-s}\mathbb{E}^{\triangleleft}_0 [J(X_s)]\\
&= \int_0^{\infty}{\rm d}s\, {\rm e}^{-s} \lim_{x\to 0} |x|^{\alpha-2\beta-d}\mathbb{P}^{\triangleright}_x (s<\kappa_{\{0\}}),
\end{align*}
and, thanks to \eqref{2sidedbound}, the right hand side either converges or explodes depending on whether $d<2(\alpha-\beta)$. So, remember that by Campbell's theorem, the sum of the lengths $\sum_{s\leq t} \zeta_s$ is finite a.s. for any $t>0,$ if and only if $\mathbf{N}^{\triangleright}(1-{\rm e}^{-\zeta})<\infty,$ which is equivalent to $d<2(\alpha-\beta).$ This justifies our comment following the statement of Theorem~\ref{recexctsat0}.
%
%
%
%
\section*{Acknowledgements}
We are grateful to two anonymous referees for their remarks which led to an improvement in an earlier draft of this paper.

\bibliography{references}{}

\begin{thebibliography}{10}

\bibitem{ACGZ}
L~Alili, L.~Chaumont, P.~Graczyk, and T.~\.Zak.
\newblock Inversion, duality and {D}oob {$h$}-transforms for self-similar
  {M}arkov processes.
\newblock {\em Electron. J. Probab.}, 22:Paper No. 20, 18, 2017.

\bibitem{Alsmeyer1994}
G.~Alsmeyer.
\newblock On the {M}arkov renewal theorem.
\newblock {\em Stochastic Process. Appl.}, 50(1):37--56, 1994.

\bibitem{Alsmeyer2014}
G.~Alsmeyer.
\newblock Quasistochastic matrices and {M}arkov renewal theory.
\newblock {\em J. Appl. Probab.}, 51A(Celebrating 50 Years of The Applied
  Probability Trust):359--376, 2014.

\bibitem{AsmussenQueue}
S.~Asmussen.
\newblock {\em Applied probability and queues}, volume~51 of {\em Applications
  of Mathematics (New York)}.
\newblock Springer-Verlag, New York, second edition, 2003.
\newblock Stochastic Modelling and Applied Probability.

\bibitem{AA}
S.~Asmussen and H.~Albrecher.
\newblock {\em Ruin probabilities}, volume~14 of {\em Advanced Series on
  Statistical Science \& Applied Probability}.
\newblock World Scientific Publishing Co. Pte. Ltd., Hackensack, NJ, second
  edition, 2010.

\bibitem{BB}
R.~Ba\~nuelos and K.~Bogdan.
\newblock Symmetric stable processes in cones.
\newblock {\em Potential Anal.}, 21(3):263--288, 2004.

\bibitem{BDS}
Rodrigo Ba\~{n}uelos and Robert~G. Smits.
\newblock Brownian motion in cones.
\newblock {\em Probab. Theory Related Fields}, 108(3):299--319, 1997.

\bibitem{BBCK}
J.~Bertoin, T.~Budd, N.~Curien, and I.~Kortchemski.
\newblock Martingales in self-similar growth-fragmentations and their
  connections with random planar maps.
\newblock {\em Probab. Theory Related Fields}, 172(3-4):663--724, 2018.

\bibitem{BC02}
J.~Bertoin and M.-E. Caballero.
\newblock Entrance from {$0+$} for increasing semi-stable {M}arkov processes.
\newblock {\em Bernoulli}, 8(2):195--205, 2002.

\bibitem{BCK}
J.~Bertoin, N.~Curien, and I.~Kortchemski.
\newblock Random planar maps and growth-fragmentations.
\newblock {\em Ann. Probab.}, 46(1):207--260, 2018.

\bibitem{BS}
J.~Bertoin and M.~Savov.
\newblock Some applications of duality for {L}\'evy processes in a half-line.
\newblock {\em Bull. Lond. Math. Soc.}, 43(1):97--110, 2011.

\bibitem{BWat}
J.~Bertoin and A.~R. Watson.
\newblock A probabilistic approach to spectral analysis of growth-fragmentation
  equations.
\newblock {\em J. Funct. Anal.}, 274(8):2163--2204, 2018.

\bibitem{BY02}
J.~Bertoin and M.~Yor.
\newblock The entrance laws of self-similar {M}arkov processes and exponential
  functionals of {L}\'evy processes.
\newblock {\em Potential Anal.}, 17(4):389--400, 2002.

\bibitem{Weyl}
P.~Biane, P.~Bougerol, and N.~O'Connell.
\newblock Littelmann paths and {B}rownian paths.
\newblock {\em Duke Math. J.}, 130(1):127--167, 2005.

\bibitem{Weyl4}
P.~Biane, P.~Bougerol, and N.~O'Connell.
\newblock Continuous crystal and {D}uistermaat-{H}eckman measure for {C}oxeter
  groups.
\newblock {\em Adv. Math.}, 221(5):1522--1583, 2009.

\bibitem{blum}
L.~E. Blumenson.
\newblock Classroom {N}otes: {A} {D}erivation of {$n$}-{D}imensional
  {S}pherical {C}oordinates.
\newblock {\em Amer. Math. Monthly}, 67(1):63--66, 1960.

\bibitem{Blum1983}
R.~M. Blumenthal.
\newblock On construction of {M}arkov processes.
\newblock {\em Z. Wahrsch. Verw. Gebiete}, 63(4):433--444, 1983.

\bibitem{BGR}
R.~M. Blumenthal, R.~K. Getoor, and D.~B. Ray.
\newblock On the distribution of first hits for the symmetric stable processes.
\newblock {\em Trans. Amer. Math. Soc.}, 99:540--554, 1961.

\bibitem{Bogdan97}
K.~Bogdan.
\newblock The boundary {H}arnack principle for the fractional {L}aplacian.
\newblock {\em Studia Math.}, 123(1):43--80, 1997.

\bibitem{BBC}
K.~Bogdan, K.~Burdzy, and Z.-Q. Chen.
\newblock Censored stable processes.
\newblock {\em Probab. Theory Related Fields}, 127(1):89--152, 2003.

\bibitem{BPW}
K.~Bogdan, Z.~Palmowski, and L.~Wang.
\newblock Yaglom limit for stable processes in cones.
\newblock {\em Electron. J. Probab.}, 23:Paper No. 11, 19, 2018.

\bibitem{BZ}
K.~Bogdan and T.~{\.Z}ak.
\newblock On {K}elvin transformation.
\newblock {\em J. Theoret. Probab.}, 19(1):89--120, 2006.

\bibitem{CC06b}
M.~E. Caballero and L.~Chaumont.
\newblock Conditioned stable {L}\'evy processes and the {L}amperti
  representation.
\newblock {\em J. Appl. Probab.}, 43(4):967--983, 2006.

\bibitem{CC06a}
M.~E. Caballero and L.~Chaumont.
\newblock Weak convergence of positive self-similar {M}arkov processes and
  overshoots of {L}\'evy processes.
\newblock {\em Ann. Probab.}, 34(3):1012--1034, 2006.

\bibitem{CKPR12}
L.~Chaumont, A.~E. Kyprianou, J.~C. Pardo, and V.~M. Rivero.
\newblock Fluctuation theory and exit systems for positive self-similar
  {M}arkov processes.
\newblock {\em Ann. Probab.}, 40(1):245--279, 2012.

\bibitem{CPR}
L.~Chaumont, H.~Pant\'{i}, and V.~M. Rivero.
\newblock The {L}amperti representation of real-valued self-similar {M}arkov
  processes.
\newblock {\em Bernoulli}, 19(5B):2494--2523, 2013.

\bibitem{Cinlar}
E.~{\c{C}}inlar.
\newblock Markov additive processes. {I}, {II}.
\newblock {\em Z. Wahrscheinlichkeitstheorie und Verw. Gebiete}, 24:85--93;
  ibid. 24 (1972), 95--121, 1972.

\bibitem{Cinlar2}
E.~{\c{C}}inlar.
\newblock L\'evy systems of {M}arkov additive processes.
\newblock {\em Z. Wahrscheinlichkeitstheorie und Verw. Gebiete}, 31:175--185,
  1974/75.

\bibitem{Cinlar1}
E.~{\c{C}}inlar.
\newblock Entrance-exit distributions for {M}arkov additive processes.
\newblock {\em Math. Programming Stud.}, (5):22--38, 1976.
\newblock Stochastic systems: modeling, identification and optimization, I
  (Proc. Sympos., Univ. Kentucky, Lexington, Ky., 1975).

\bibitem{DeB}
R.~D. DeBlassie.
\newblock The first exit time of a two-dimensional symmetric stable process
  from a wedge.
\newblock {\em Ann. Probab.}, 18(3):1034--1070, 1990.

\bibitem{DM}
C.~Dellacherie and P.-A. Meyer.
\newblock {\em Probabilit\'es et potentiel. {C}hapitres {XII}--{XVI}}.
\newblock Publications de l'Institut de Math\'ematiques de l'Universit\'e de
  Strasbourg [Publications of the Mathematical Institute of the University of
  Strasbourg], XIX. Hermann, Paris, second edition, 1987.
\newblock Th\'eorie du potentiel associ\'ee \`a une r\'esolvante. Th\'eorie des
  processus de Markov. [Potential theory associated with a resolvent. Theory of
  Markov processes], Actualit\'es Scientifiques et Industrielles [Current
  Scientific and Industrial Topics], 1417.

\bibitem{DDK}
S.~Dereich, K.~D\"oring, and A.~E. Kyprianou.
\newblock Real self-similar processes started from the origin.
\newblock {\em Ann. Probab.}, 45(3):1952--2003, 2017.

\bibitem{DK}
L.~D\"{o}ring and A.~E. Kyprianou.
\newblock Entrance and exit at infinity for stable jump diffusions.
\newblock {\em Ann. Probab.}, 48(3):1220--1265, 2020.

\bibitem{DKW}
L.~D\"{o}ring, A.~E. Kyprianou, and P.~Weissmann.
\newblock Stable processes conditioned to avoid an interval.
\newblock {\em Stochastic Process. Appl.}, 130(2):471--487, 2020.

\bibitem{Weyl2}
Y.~Doumerc and N.~O'Connell.
\newblock Exit problems associated with finite reflection groups.
\newblock {\em Probab. Theory Related Fields}, 132(4):501--538, 2005.

\bibitem{Fitz06}
P.~J. Fitzsimmons.
\newblock On the existence of recurrent extensions of self-similar {M}arkov
  processes.
\newblock {\em Electron. Comm. Probab.}, 11:230--241, 2006.

\bibitem{fitzsimmons-getoor}
P.~J. Fitzsimmons and R.~K. Getoor.
\newblock Excursion theory revisited.
\newblock {\em Illinois J. Math.}, 50(1-4):413--437, 2006.

\bibitem{JS}
J.~Jacod and A.~N. Shiryaev.
\newblock {\em Limit theorems for stochastic processes}, volume 288 of {\em
  Grundlehren der Mathematischen Wissenschaften [Fundamental Principles of
  Mathematical Sciences]}.
\newblock Springer-Verlag, Berlin, second edition, 2003.

\bibitem{Kaspi}
H.~Kaspi.
\newblock On the symmetric {W}iener-{H}opf factorization for {M}arkov additive
  processes.
\newblock {\em Z. Wahrsch. Verw. Gebiete}, 59(2):179--196, 1982.

\bibitem{kaspi-excursions}
H.~Kaspi.
\newblock On invariant measures and dual excursions of {M}arkov processes.
\newblock {\em Z. Wahrsch. Verw. Gebiete}, 66(2):185--204, 1984.

\bibitem{kestenMAP}
H.~Kesten.
\newblock Renewal theory for functionals of a {M}arkov chain with general state
  space.
\newblock {\em Ann. Probability}, 2:355--386, 1974.

\bibitem{KKPW}
A.~Kuznetsov, A.~E. Kyprianou, J.~C. Pardo, and A.~R. Watson.
\newblock The hitting time of zero for a stable process.
\newblock {\em Electron. J. Probab.}, 19:no. 30, 26, 2014.

\bibitem{Deep1}
A.~E. Kyprianou.
\newblock Deep factorisation of the stable process.
\newblock {\em Electron. J. Probab.}, 21:Paper No. 23, 28, 2016.

\bibitem{ALEAKyp}
A.~E. Kyprianou.
\newblock Stable {L}\'{e}vy processes, self-similarity and the unit ball.
\newblock {\em ALEA Lat. Am. J. Probab. Math. Stat.}, 15(1):617--690, 2018.

\bibitem{KPW}
A.~E. Kyprianou, J.~C. Pardo, and A.~R. Watson.
\newblock Hitting distributions of {$\alpha$}-stable processes via path
  censoring and self-similarity.
\newblock {\em Ann. Probab.}, 42(1):398--430, 2014.

\bibitem{Deep2}
A.~E. Kyprianou, V.~M. Rivero, and B.~\c{S}eng\"{u}l.
\newblock Deep factorisation of the stable process {II}: {P}otentials and
  applications.
\newblock {\em Ann. Inst. Henri Poincar\'{e} Probab. Stat.}, 54(1):343--362,
  2018.

\bibitem{Deep3}
A.~E. Kyprianou, V.~M. Rivero, and W.~Satitkanitkul.
\newblock Deep factorisation of the stable process {III}: radial excursion
  theory and the point of closest reach.
\newblock {\it To appear in Potential Analysis}., 2016.

\bibitem{KRS}
A~E. Kyprianou, V.~M. Rivero, and W.~Satitkanitkul.
\newblock Conditioned real self-similar {M}arkov processes.
\newblock {\em Stochastic Process. Appl.}, 129(3):954--977, 2019.

\bibitem{KV}
A.~E. Kyprianou and S.~M. Vakeroudis.
\newblock Stable windings at the origin.
\newblock {\em Stochastic Process. Appl.}, 128(12):4309--4325, 2018.

\bibitem{lalley}
S.~P. Lalley.
\newblock Conditional {M}arkov renewal theory. {I}. {F}inite and denumerable
  state space.
\newblock {\em Ann. Probab.}, 12(4):1113--1148, 1984.

\bibitem{L72}
J.~Lamperti.
\newblock Semi-stable {M}arkov processes. {I}.
\newblock {\em Z. Wahrscheinlichkeitstheorie und Verw. Gebiete}, 22:205--225,
  1972.

\bibitem{M75}
B.~Maisonneuve.
\newblock Exit systems.
\newblock {\em Ann. Probability}, 3(3):399--411, 1975.

\bibitem{MH}
P.~J. M\'{e}ndez-Hern\'{a}ndez.
\newblock Exit times from cones in {$\bold R^n$} of symmetric stable processes.
\newblock {\em Illinois J. Math.}, 46(1):155--163, 2002.

\bibitem{Nagasawa1964}
M.~Nagasawa.
\newblock Time reversions of {M}arkov processes.
\newblock {\em Nagoya Math. J.}, 24:177--204, 1964.

\bibitem{Weyl3}
N.~O'Connell.
\newblock Random matrices, non-colliding processes and queues.
\newblock In {\em S\'eminaire de {P}robabilit\'es, {XXXVI}}, volume 1801 of
  {\em Lecture Notes in Math.}, pages 165--182. Springer, Berlin, 2003.

\bibitem{R05}
V.~M. Rivero.
\newblock Recurrent extensions of self-similar {M}arkov processes and
  {C}ram\'er's condition.
\newblock {\em Bernoulli}, 11(3):471--509, 2005.

\bibitem{R07}
V.~M. Rivero.
\newblock Recurrent extensions of self-similar {M}arkov processes and
  {C}ram\'er's condition. {II}.
\newblock {\em Bernoulli}, 13(4):1053--1070, 2007.

\bibitem{Salt1986a}
T.~S. Salisbury.
\newblock Construction of right processes from excursions.
\newblock {\em Probab. Theory Related Fields}, 73(3):351--367, 1986.

\bibitem{Salt1986b}
T.~S. Salisbury.
\newblock On the {I}t\^o excursion process.
\newblock {\em Probab. Theory Related Fields}, 73(3):319--350, 1986.

\bibitem{Sato}
K-I Sato.
\newblock {\em L\'evy processes and infinitely divisible distributions},
  volume~68 of {\em Cambridge Studies in Advanced Mathematics}.
\newblock Cambridge University Press, Cambridge, 2013.
\newblock Translated from the 1990 Japanese original, Revised edition of the
  1999 English translation.

\bibitem{StR}
R.~Stephenson.
\newblock On the exponential functional of {M}arkov additive processes, and
  applications to multi-type self-similar fragmentation processes and trees.
\newblock {\em ALEA Lat. Am. J. Probab. Math. Stat.}, 15(2):1257--1292, 2018.

\end{thebibliography}
\bibliographystyle{plain}

\end{document}